\documentclass[11pt]{article}

\usepackage[pdftex]{graphicx}
\usepackage{amsmath,amssymb}
\usepackage[usenames,dvipsnames,pdftex]{color}
\usepackage{epsfig}
\usepackage{amssymb,amsmath, amsthm,epsfig,multirow}

\newtheorem{defn}{Definition}[section]
\newtheorem{theorem}{Theorem}[section]
\newtheorem{corollary}[theorem]{Corollary}
\newtheorem{remark}{Remark}

\begin{document}
\makeatletter{\renewcommand*{\@makefnmark}{}
\title{\bf On a regularization-correction approach for the approximation of piecewise smooth functions}

\footnotetext{This work was funded by project 20928/PI/18 (Proyecto financiado por la Comunidad Aut\'onoma de la Regi\'on de Murcia a trav\'es de la convocatoria de Ayudas a proyectos para el desarrollo de investigaci\'on cient\'ifica y t\'ecnica por grupos competitivos, incluida en el Programa Regional de Fomento de la Investigaci\'on Cient\'ifica y T\'ecnica (Plan de Actuaci\'on 2018) de la Fundaci\'on S\'eneca-Agencia de Ciencia y Tecnolog\'ia de la Regi\'on de Murcia) and by the Spanish national research project PID2019-108336GB-I00.}

\date{}
\author{Sergio Amat,
\thanks{
 Departamento de Matem\'atica Aplicada y Estad\'{\i}stica.
   Universidad  Polit\'ecnica de Cartagena (Spain).
e-mail:{\tt sergio.amat@upct.es}}\and
David Levin,
\thanks{
School of Mathematical Sciences. Tel-Aviv University, Tel-Aviv (Israel).
 e-mail:{\tt levindd@gmail.com}
}
\and Juan Ruiz \thanks{ Departamento de Matem\'atica Aplicada y Estad\'{\i}stica.
   Universidad  Polit\'ecnica de Cartagena (Spain).
e-mail:{\tt juan.ruiz@upct.es}}
}
\maketitle

\begin{abstract}
Linear approximation approaches suffer from Gibbs oscillations when approximating functions with singularities. ENO-SR resolution is a local approach avoiding oscillations and with a full order of accuracy, but a loss of regularity
of the approximant appears.
The goal of this paper is to introduce a new approach having both properties of full accuracy and regularity. In order to obtain it, we propose a three-stage algorithm: first, the data is smoothed by subtracting an appropriate non-smooth data sequence; then a chosen high order linear approximation operator is applied to the smoothed data and finally, an approximation with the proper singularity structure is reinstated by correcting the smooth approximation with the non-smooth element used in the first stage.
We apply this approach to both cases of point-value data and of cell-average data, using the 4-point subdivision algorithm in the second stage. Using the proposed approach we are able to construct approximations with high precision, with high piecewise regularity, and without diffusion nor oscillations in the presence of discontinuities.
\end{abstract}


\section{Introduction}

The ENO (Essential Non-Oscillatory) reconstruction is an approximation method for non-smooth data \cite{HARTEN-original}. It uses a local linear approximation on a selection of the stencils that aims to avoid zones affected by singularities. Thus, the precision is only reduced at the intervals which contain a singularity. The ENO subcell resolution (ENO-SR) technique improves the approximation properties of the reconstruction
even at intervals that contain a singularity. In \cite{ACDD}, the authors proposed a rigorous analysis of the ENO-SR procedure both for point-value data and for cell-average sampling.  For a piecewise $C^2$ function $f$ with a jump $[f']$ of the first derivative at $x^*$, the authors proved that the singularity is
always detected for a grid-spacing discretization $h$ smaller than the critical scale $h_c$ being,
\begin{equation}\label{hc2}
h_c :=\frac{|[f^{'}]|}{4\sup_{t\in\mathbb{R}\backslash\{x^*\}} |f^{''}(t)|}.
\end{equation}
This critical scale represents the minimal level of resolution needed to
distinguish between the singularity and a smooth region.

The detection algorithms in \cite{ACDD} can detect corner and jump discontinuities. We note that only corner singularities can be located using the
point-values discretization, as the location of jumps in the function is lost during the discretization process. Jumps in the function can be located if we use other kinds of discretization that accumulate information over the whole interval of discretization instead of sampling information at isolated locations. One example is the cell-average discretization, that we introduce later on in Section \ref{discretizacion}. In \cite{ACDD}, the authors also perform an analysis for jump discontinui\-ties and data discretized by cell-averages, obtaining the critical scale,
\begin{equation}\label{hc}
h_c :=\frac{|[f]|}{4\sup_{t\in\mathbb{R}\backslash\{x^*\}} |f^{'}(t)|},
\end{equation}
where $[f]$ is the jump in the function $f$.

The algorithm in \cite{ACDD} for the point-values sampling is based on the comparison of second-order differences for the detection of intervals that are potential candidates of containing a singularity.
Following \cite{ACDD}, we can assume that given a continuous piecewise functions there exists a critical scale $h_c$ depending only on derivatives of $f$ at smooth regions such that all the intervals with a singularity are detected.
In the case of a corner singularity, once the singularity interval has been detected, we build polynomials to the left and to the right of the suspicious
interval. Then, we solve for the root of the polynomial $H$ which is the difference between the two polynomials. As proved in
\cite{ACDD}, for a small enough grid-spacing, there is a unique root of this function inside the suspicious interval. Thus, a good approximation for the position of the discontinuity can be easily obtained by finding the roots of $H$. The interested reader can refer to \cite{Arandiga:2003:AII:641932.641950} for a nice discussion about the process. The accuracy of the results depends on the order of the polynomials, as the authors prove in Lemma 3, statement 3 of \cite{ACDD}. In the cell-average setting, jumps in the function transform into corner singularities in the primitive and their position can be easily obtained following the process described above.

On the other hand, in \cite{BL}, a specific prediction operator in the interpolatory
framework was proposed and analyzed. It is oriented towards the representation
of piecewise smooth continuous functions. Given a family of singularity points, a 4-point linear position-dependent prediction
operator is defined. The convergence of this non-uniform subdivision process is esta\-blished using matrix formalism.
The algorithm leads to the successful control of the classical
Gibbs phenomenon associated to the approximation of locally discontinuous functions. This position-dependent algorithm is equivalent to the
ENO-SR subdivision scheme assuming the previous knowledge of the singularity positions, in particular improving the accuracy of the linear a\-pproach.

Despite their good accuracy properties, the theoretical regularity of ENO and ENO-SR schemes in \cite{BL} is $C^{1-}$ in the point-values, which is smaller than the regularity of the 4-point linear subdivision scheme, that is $C^{2-}$, as we can see in \cite{CDM}. We emphasize that $C^2$ regularity is important in aerodynamics and hydrodynamics applications which motivated the ENO algorithm from its very outset. Other adaptive interpolation schemes like the PPH algorithm \cite{PPH} manage to eliminate the Gibbs phenomenon close to the discontinuities, but introducing diffusion and, thus, reducing the order of accuracy near discontinuities but also the regularity (that is $C^{1-}$ numerically) close to the discontinuities.

The goal of this paper is to develop a regularization-correction approach (RC) to the problem. In the first stage, the data is smoothed by subtracting an appropriate non-smooth data sequence. Then a uniform linear 4-point subdivision approximation operator is applied to the smoothed data. Finally, an approximation with the proper singularity structure is restored by correcting the smooth approximation with the non-smooth element used in the first stage.
Indeed, we prove that the suggested RC procedure produces approximations for functions with discontinuities which have the following five important properties:

1. Interpolation.

2. High precision.

3. High piecewise regularity.

4. No diffusion.

5. No oscillations

As far as we know, this is the first time that a procedure that owns all these properties at the same time appears in the literature.
We use the particular case of the 4-point Dubuc-Deslauriers interpolatory subdivision scheme \cite{DD} through which we obtain a $C^{2-}$ piecewise regular limit function and that is capable of reproducing piecewise cubic polynomials.
Indeed, it is straightforward to obtain higher regularity and reproduction of piecewise polynomials of a higher degree, just using the same technique with larger stencils.

The paper is organized as follows: in Section \ref{discretizacion} we introduce the discretizations that is being used in the paper, Section \ref{nuevo} is dedicated to introducing our RC approximation approach, and the analysis of the approximation order of the resulting approximants. In Section \ref{numexp_point} we present some experiments for univariate and bivariate functions discretized by point-values sampling, in Section \ref{numexp_cell} we show some experiments for univariate and bivariate functions discretized by cell-averages sampling and, finally, in Section  \ref{conclusion} we expose the conclusions.

\section{The discretizations: point-values and cell-a\-verages}\label{discretizacion}

In the case of a point-values discretization we can just consider that we are given a vector of values on a grid. This means that our data is interpreted as the sampling of a function at grid-points $\{x_j=jh\}$. In this case, as it has been mentioned above, the position of the discontinuities in the function is lost during the discretization process and there is no hope to recover their exact position. Other kind of singularities, such as discontinuities in the first derivative, can be located using the point-values discretization. Another option is to consider the original data as averages of a certain piecewise continuous function over the intervals of a grid. This is the {cell-average setting}, which also allows to locate jump discontinuities in the function. In both cases we consider the grid points  in $[0, 1]$:
$$X=\{ x_{j} \}_{j=0}^{N},\quad x_{j}=jh, \quad h=\frac{1}{N}.$$
For the point-values case we use the discretization $\{f_j=f(x_j)\}_{j=0}^{N}$ at the data points $\{x_j= jh\}_{j=0}^N$.

On the other hand, for the cell-averages sampling we are given the local averages values,
\begin{equation}\label{discretization}
 \bar{f}_j=\frac{1}{h}\int^{x_j}_{x_{j-1}}f(x)dx,\quad j=1, \cdots, N.
\end{equation}
Also in this case we aim at approximating the underlying function $f$.

Let us define the sequence $\{F_j\}$ as,
\begin{equation}\label{primitive}
F_j=h\sum_{i=1}^{j}\bar{f}_i=\int_{0}^{x_j}f(y)dy,\quad j=1, \cdots, N,
\end{equation}
taking $F_0=0$. Denoting by $F$ the primitive function of $f$, i.e. $F(x)=\int_0^xf(y)dy$, the values $\{F_j\}$ are the point-values discretization of $F$. Now we are back in the case of point-value data, for $F$, and after finding an approximation $G(x)$ to $F(x)$, $g(x)=G'(x)$ would be the approximation to $f(x)=F'(x)$, such that,
\begin{equation}\label{diff}
\bar{g}_j=\frac{F_j-F_{j-1}}{h}.
\end{equation}

%
%

\section{The Regularization-Correction (RC) algorithm for non-smooth data}\label{nuevo}

We present our approach for the approximation of a function with one singular point. Later on we explain how to use it for the case of several singular points.

Let $f$ be a piecewise $C^4$-smooth function on $[0, 1]$,
with a singular point at $x^*$, and assume we are given the
vector of values $\{f(x_i)\}_{i=0}^{N}$ at the data points
$\{x_i = ih\}_{i=0}^{N}$, $N=1+1/h$. We denote by $f^-(x)$ and $f^+(x)$ the functions to the left and to the right of the discontinuity respectively.

In our framework we are going to use the 4-point Dubuc-Deslauriers subdivision scheme:
\begin{equation}\label{lineal}
\left\{\begin{array}{l}
(S f^{k})_{2j}= f_j^k,\\
(S f^{k})_{2j+1}= -\frac{1}{16} f_{j-1}^k + \frac{9}{16} f_j^k + \frac{9}{16} f_{j+1}^k- \frac{1}{16} f_{j+2}^k.
\end{array}\right.
\end{equation}
This scheme has fourth order of accuracy for $C^4$ functions, and it has
a $C^{2^-}$ regularity \cite{DD}, \cite{DLG}.
We aim at retrieving these properties for functions with jump singularities in  the first derivative and the function for data discretized by point-values or by cell-averages.

In the following sections we suggest the framework of a regularization-correction algorithm that is adapted to the discontinuities, that avoids the Gibbs phenomenon, that attains the same regularity at smooth zones, as the equivalent linear algorithm, without diffusion and that reproduces polynomials of the degree associated to the linear scheme used.

\subsection{The RC approximation algorithm for point-values data}\label{algorithm1}
In this subsection we introduce the RC approximation algorithm. The main idea is to use the given discrete data of $f$ in order to find an explicit function $q(x)$ such that $g(x)=f(x)-q(x)$ is a smooth function. Then we can apply to the data
$\{g_j=g(x_j)\}_{j=0}^{N}$ any standard approximation procedure with high smoothness and high approximation order. For example, we can use the above  4-point subdivision algorithm and we denote the resulting approximation by $g^\infty$. The last step is the correction step in which we define the approximation to $f$ as $g^\infty+q$.

Let $T_3^-(x)$ and $T_3^+(x)$ be the third order Taylor approximations of $f^-$ and of $f^+$ at $x^*$ respectively. Consider the one sided cubic polynomial
\begin{equation}\label{T}
\left\{\begin{array}{l}
T_+(x)=0,\ \ \ \ \ \ \ \ \ \ \ \ \ \ \ \ \ \ \ \ \ \ x<x^*,\\
T_+(x)=T_3^+(x)-T_3^-(x),\ \ \  x\ge x^*.
\end{array}\right.
\end{equation}
For $x \geq x^*$
\begin{equation}\label{T1}
T_+(x) = [f]+[f^{'}] (x-x^*)+\frac{1}{2}[f^{''}](x-x^*)^2+\frac{1}{6}[f^{'''}](x-x^*)^3,
\end{equation}
where $[f],\ [f^{'}],\ [f^{''}],\ [f^{'''}]$ denote the jumps in the
derivatives of $f$ at $x^*$: $[f]=f^+(x^*)-f^-(x^*), [f^{'}]=(f^+)'(x^*)-(f^-)'(x^*), [f^{''}]=(f^+)''(x^*)-(f^-)''(x^*)$ and $[f^{'''}]=(f^+)'''(x^*)-(f^-)'''(x^*).$

It follows that $$g\equiv f-T_+\in C^3[0,1],$$ and this observation is the basis for our proposed algorithm. Since we do not know the exact left and right derivatives of $f$ at $x^*$, we will use instead a fourth order approximation of $T_+(x)$:
For $x \geq x^*$,
\begin{equation}\label{tildeT}
\widetilde{T}_+(x) = \widetilde{[f]}+\widetilde{[f^{'}]}(x-x^*)+\frac{1}{2}\widetilde{[f^{''}]}(x-x^*)^2+\frac{1}{6}\widetilde{[f^{'''}]}(x-x^*)^3,
\end{equation}
where
\begin{equation}\label{order4}
\left(\begin{array}{l}
\widetilde{[f]}\\
\widetilde{[f']}\\
\widetilde{[f'']}\\
\widetilde{[f''']}
\end{array}\right)=
\left(\begin{array}{l}
[f]\\
\left[f'\right]\\
\left[f''\right]\\
\left[f'''\right]
\end{array}\right)+
\left(\begin{array}{l}
O(h^4)\\
O(h^3)\\
O(h^2)\\
O(h)
\end{array}\right).
\end{equation}

In section \ref{calc_disc}, we will propose a procedure to approximate these jumps.

Comparing equations (\ref{T1}) and (\ref{tildeT}) and using (\ref{order4}) it follows that near $x=x^*$, i.e., if $|x-x^*|=O(h)$,
\begin{equation}\label{apprT}
|T_+(x)-\tilde{T}_+(x)|=O(h^4),\ \ \ as\  h\to 0.
\end{equation}

\vspace{0.5cm}
Our algorithm has three steps:
\vspace{0.5cm}
\begin{itemize}
\item Step {\bf 1}: Smoothing the data.

We compute the new data using $\tilde{T}_+(x)$ in (\ref{tildeT}),
\begin{equation}\label{corr}
\tilde{g}(x_i) = f(x_i) -\tilde{T}_+(x_i),\ \ i=0,...,N.
\end{equation}
It is clear that $g(x)$ does not present singularities up to the third derivative. As we show below, using a fourth order approximation $\tilde{g}$ to $g$, implies a truncation error that is $O(h^4)$.
\item Step {\bf 2}: Subdivision.

We apply the 4-point interpolatory subdivision scheme to the
data $\tilde{g}(x_j)$ and we denote the limit function by $\tilde{g}^\infty$.

\item Step {\bf 3}: Correcting the approximation.

\begin{equation}\label{T2}
\tilde{f}^\infty (x) = \tilde{g}^\infty (x) + \widetilde{T}_+(x).
\end{equation}
\end{itemize}

In Figure \ref{ejemplo_corr} we present an example with one of the functions that we use in the numerical experiments (more specifically, the one in (\ref{exp1})). The circles represent a function discretized by point-values with a discontinuity in the first derivative. The dots represent the smoothed data in (\ref{corr}), and we can see that it owns better regularity properties than the original data. Using the strategy presented in this section, we apply a subdivision scheme to the corrected data and then we compute the corrected approximation in (\ref{T2}).
\begin{figure}[!ht]
\centerline{\psfig{figure=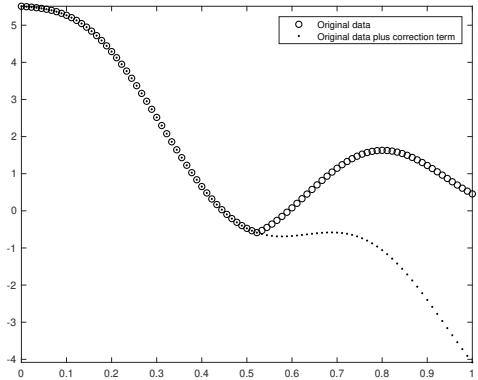,height=5cm}}
\caption{The circles in this graphs correspond to the original data obtained from the discretization of the function in (\ref{exp1}). The dots correspond to the corrected data in (\ref{corr}).}\label{ejemplo_corr}
\end{figure}

\begin{remark}\label{2omas}
It is important to note here that the algorithm can be applied to functions with $m>1$ discontinuities. In this case the correction can be applied iteratively from left to right as the discontinuities are found. The total correction term would be,
$$\tilde{T}_+^{\textrm{total}}(x)=\sum_{n=1}^{m}\tilde{T}_+^{[n]}(x),$$
being
\begin{equation}
\widetilde{T}_+^{[n]}(x) = \widetilde{[f]_n}+\widetilde{[f^{'}]}_n(x-x^*_n)+\frac{1}{2}\widetilde{[f^{''}]}_n(x-x^*_n)^2+\frac{1}{6}\widetilde{[f^{'''}]}_n(x-x^*_n)^3,
\end{equation}
the correction term (\ref{tildeT}) at each of the $m$ discontinuities placed at $x^*_n, n=1, \cdots, m$, found in the data. In this case, $\tilde{T}_+^{\textrm{total}}(x)$ should be used instead on $\tilde{T}_+(x)$ in steps 1 and 3 of the algorithm.
\end{remark}

\subsection{The RC approximation algorithm for cell-averages data}\label{algorithm2}

In order to work with data discretized by cell-averages, $\{\bar{f}_j\}$, we firstly obtain the point-values $\{F_j\}$ of the primitive function $F$ (\ref{primitive}), as in the original Harten's framework \cite{HARTEN-original}. Then we apply our strategy for the point-values discretization described above to obtain an approximation $G\sim F$ and then approximate $f$ by $G'$. The main difference from the case of point-values data is that the function $F$ is a continuous function, with a jump in its first derivative. Thus, an additional step in the algorithm for cell-averages data is to define $x^*$ as the intersection point of the two local polynomials derived from the data $\{F_j\}$. In Figure \ref{paso_primitiva} we sketch the overall process.

\begin{figure}[!ht]
\centerline{\psfig{figure=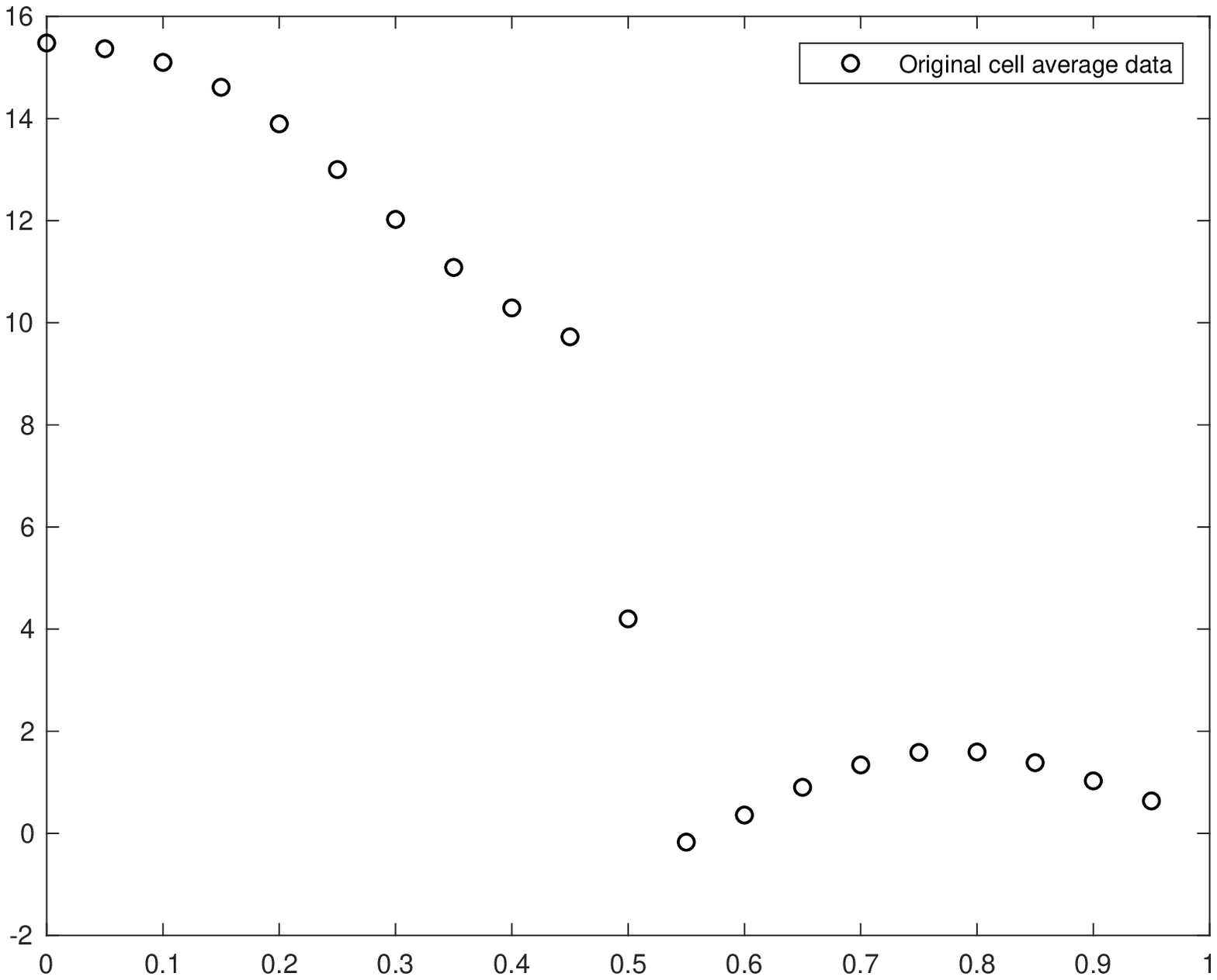,height=4cm}\\
\psfig{figure=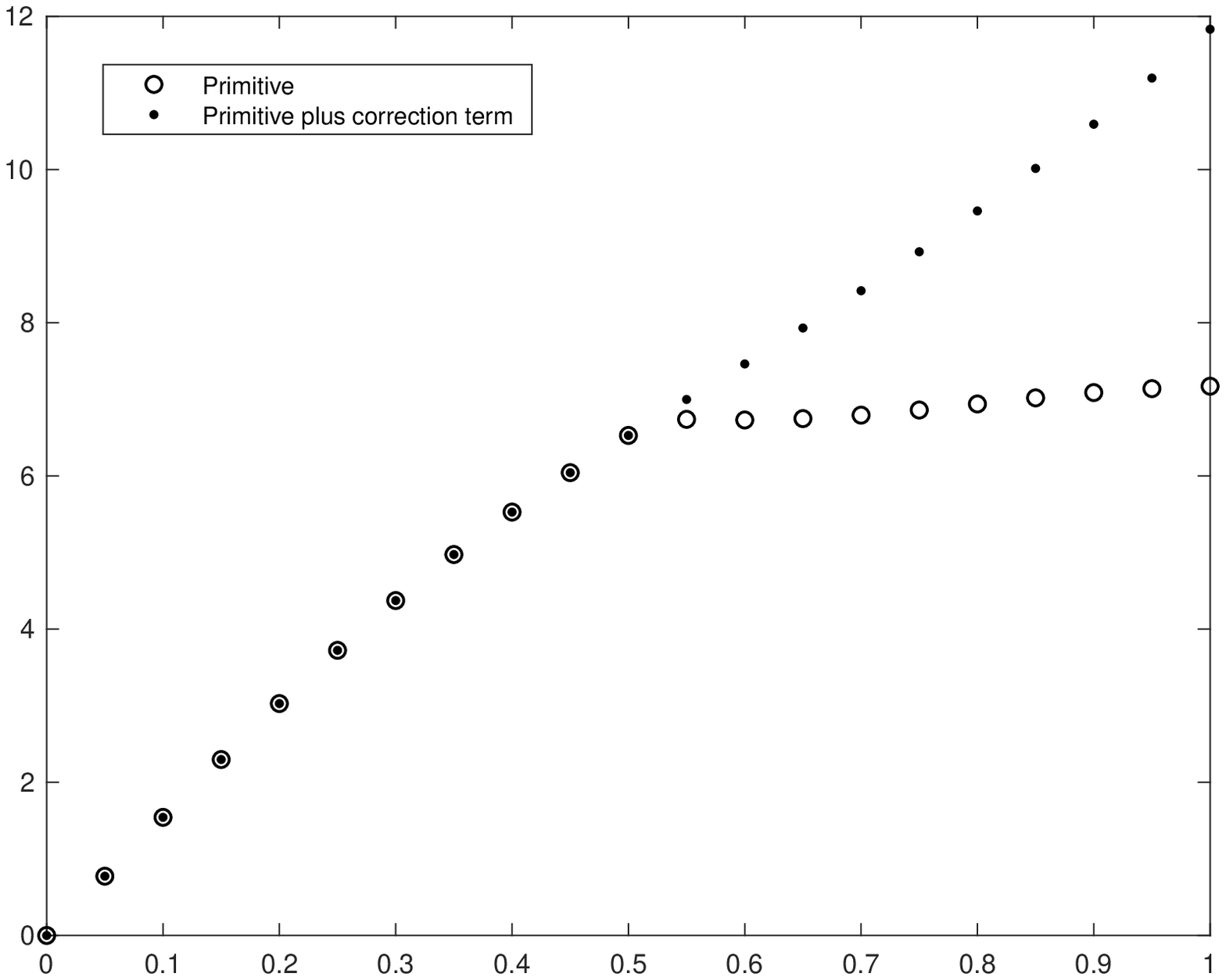,height=4cm}\\
\psfig{figure=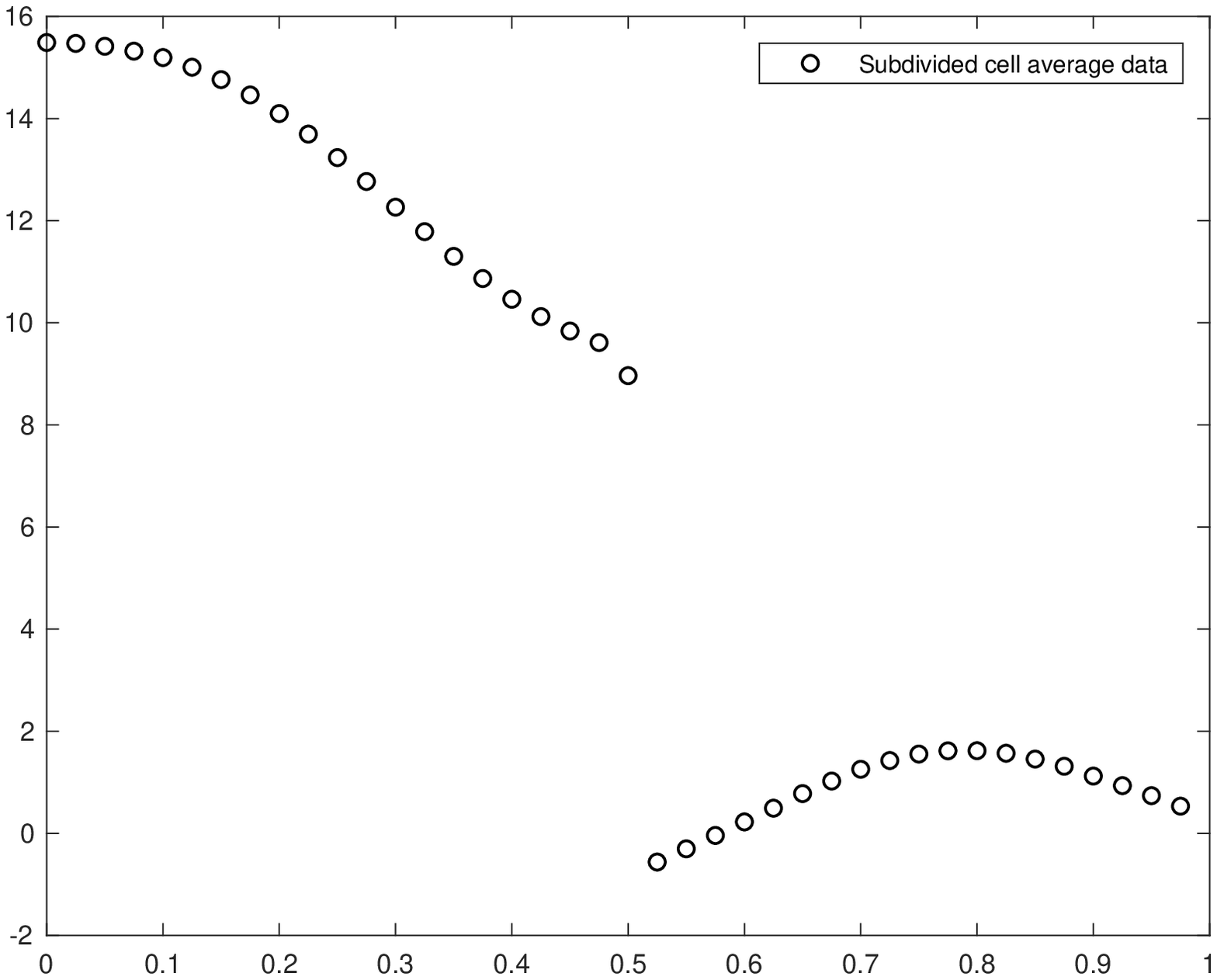,height=4
cm}}
\caption{Left, cell-averages of the function presented in (\ref{exp3}). Center, primi\-tive of the data in the plot to the left, and the corrected primitive. Right, result of the RC approximation to the function discretized by cell-averages.}\label{paso_primitiva}
\end{figure}

In the next section we present the construction of a fourth order accurate approximation of the jumps in the function and its derivatives for data discretized by point-values.

\subsection{Fourth-order approximation of the jumps in the function and its derivatives in the point-values framework}\label{calc_disc}

\medskip
The approximation of the jumps with the desired accuracy can be obtained using the strategy derived in \cite{ALR}.
In order to obtain the approximate jumps in the derivatives of $f$ we need to know the position of the discontinuity $x^*$ up to certain accuracy. In the introduction we mentioned how to obtain an approximation of the position of the discontinuity $x^*$ with the desired accuracy following \cite{ACDD}. If we work with stencils of four points, we need $O(h^4)$ accuracy in the detection. Let us suppose that we know that the discontinuity is placed at a distance $\alpha$ from $x_j$ in the interval $\{x_j, x_{j+1}\}$. If we want $O(h^4)$ accu\-racy for the approximation of the jump relations, we need to use four points to each side of the discontinuity. Let's suppose that we use the data $\{f_{j-3}, f_{j-2}, f_{j-1}, f_{j}, f_{j+1}, f_{j+2}, f_{j+3}, f_{j+4}\}$ placed at the positions $\{x_{j-3}, x_{j-2}, x_{j-1}, x_{j}, x_{j+1}, x_{j+2}, x_{j+3}, x_{j+4}\}$. Then, we can appro\-ximate the values of the derivatives of $f$ from both sides of the discontinuity just using third order Taylor expansion around $x^*$ and setting the following system of equations:
For the left hand side,
\begin{equation}\label{sys1}
\begin{aligned}
f_j&=f^{-}(x^*)-f^{-}_x(x^*) \alpha +\frac{1}{2} f^{-}_{xx}(x^*) \alpha^2-\frac{1}{3!} f^{-}_{xxx}(x^*) \alpha^3,\\
f_{j-1}&=f^{-}(x^*)-f^{-}_x(x^*) (h+\alpha) +\frac{1}{2} f^{-}_{xx}(x^*) (h+\alpha)^2-\frac{1}{3!} f^{-}_{xxx}(x^*) (h+\alpha)^3,\\
f_{j-2}&=f^{-}(x^*)-f^{-}_x(x^*) (2h+\alpha) +\frac{1}{2} f^{-}_{xx}(x^*) (2h+\alpha)^2-\frac{1}{3!} f^{-}_{xxx}(x^*) (2h+\alpha)^3,\\
f_{j-3}&=f^{-}(x^*)-f^{-}_x(x^*) (3h+\alpha) +\frac{1}{2} f^{-}_{xx}(x^*) (3h+\alpha)^2-\frac{1}{3!} f^{-}_{xxx}(x^*) (3h+\alpha)^3,
\end{aligned}
\end{equation}
and for the right hand side,
\begin{equation}\label{sys2}
\begin{aligned}
f_{j+1}&=f^{+}(x^*)+f^{+}_x(x^*) (h-\alpha) +\frac{1}{2} f^{+}_{xx}(x^*) (h-\alpha)^2+\frac{1}{3!} f^{+}_{xxx}(x^*) (h-\alpha)^3,\\
f_{j+2}&=f^{+}(x^*)+f^{+}_x(x^*) (2h-\alpha) +\frac{1}{2} f^{+}_{xx}(x^*) (2h-\alpha)^2+\frac{1}{3!} f^{+}_{xxx}(x^*) (2h-\alpha)^3,\\
f_{j+3}&=f^{+}(x^*)+f^{-}_x(x^*) (3h-\alpha) +\frac{1}{2} f^{+}_{xx}(x^*) (3h-\alpha)^2+\frac{1}{3!} f^{+}_{xxx}(x^*) (3h-\alpha)^3,\\
f_{j+4}&=f^{+}(x^*)+f^{+}_x(x^*) (4h-\alpha) +\frac{1}{2} f^{+}_{xx}(x^*) (4h-\alpha)^2+\frac{1}{3!} f^{+}_{xxx}(x^*) (4h-\alpha)^3,
\end{aligned}
\end{equation}
where $\alpha=x^*-x_{j}$. It is clear that the previous systems can be written in matrix form where the system matrix is a Vandermonde matrix, that is always invertible.

Solving the two systems (\ref{sys1}) and (\ref{sys2}), where $f^{-}(x^*), f^{-}_{x}(x^*), f^{-}_{xx}(x^*), f^{-}_{xxx}(x^*)$ and $f^{+}(x^*), f^{+}_{x}(x^*), f^{+}_{xx}(x^*), f^{+}_{xxx}(x^*)$ are the unknowns, we can obtain appro\-ximations for the jumps in the derivatives of $f$.

It is proved in \cite{ALR} that
\begin{equation}\label{acc_corr1}
\left(\begin{array}{l}
\widetilde{[f]}\\
\widetilde{[f']}\\
\widetilde{[f'']}\\
\widetilde{[f''']}
\end{array}\right)=
\left(\begin{array}{l}
f^{+}(x^*)-f^{-}(x^*)\\
f^{+}_x(x^*)-f^{-}_x(x^*)\\
f^{+}_{xx}(x^*)-f^{-}_{xx}(x^*)\\
f^{+}_{xxx}(x^*)-f^{-}_{xxx}(x^*)
\end{array}\right)+
\left(\begin{array}{l}
O(h^4)\\
O(h^3)\\
O(h^2)\\
O(h)
\end{array}\right).
\end{equation}

\subsection{The approximation properties}
Let us analyze the approximation properties of the RC scheme for the case of point-values data and for cell-averages data.

For the point-values' case we consider that a discontinuity is placed in the interval $(x_j, x_{j+1})$. Also, we consider the two cases, jumps  in the function or jumps in its first derivative. Let us start by the first case.

When dealing with jump discontinuities using point-values of a function at a certain discretization resolution, there is no way to find a high order approximation to the exact location of the discontinuity. We can only hope to locate the interval containing the jump in the function, and we can assume that the singularity is at any point $x^*$ within this interval.

\begin{remark}
It is important to note that the technique described in this and previous sections can also be applied to treat the approximation near the boundary points $0$ and $1$. Outside the boundary we use a zero padding strate\-gy, and we treat each boundary as a discontinuity which position is known. Thus, treating the boundaries is just considering more than one discontinui\-ty in the data, as explained in Remark \ref{2omas}. Using this strategy enables us to extend the approximation results discussed below to the entire interval $[0,1]$. In the following, we implicitly assume that the boundaries are treated as above.
\end{remark}

\begin{theorem}{\bf - The point-values' case: { jump discontinuities.}}\label{Theorem1}
Let $f$ be a piecewise $C^4$-smooth function on $[0, 1]$,
with a jump discontinuity, and assume we are given the
vector of values $\{f_j\}_{j=0}^{N}$ at the data points
$\{x_j = jh\}_{j=0}^{N}$, $N=1/h+1$. Let's { choose} the singular point to be { at} $x^*\in (x_j,x_{j+1})$. Then the a\-ppro\-ximation obtained by the RC algorithm interpolates the data points, it is $C^{2-}\{[0,x^*)\cup (x^*,1]\}$, and $||f-\tilde{f}^\infty||_{\infty,[0,1]\setminus \{x^*\}} = O(h^4)$ as $h\to 0${, for any piecewise $C^4$-smooth $f$ that has a singular point at $x^*$ and that has as point-values $\{f(x_j)\}_{j=0}^{N}=\{f_j\}_{j=0}^{N}$.}
\end{theorem}
\begin{proof}
The 4-point subdivision scheme is an interpolatory scheme and the limit function it defines is $C^{2-}$ \cite{DD}. Hence,
$$\tilde{g}^\infty(x_j)=f(x_j)-\tilde{T}_+(x_j),$$
and  $\tilde{g}^\infty\in C^{2-}[0,1]$.
Adding the correction term in (\ref{T2}) yields the interpolation to $f$ as
$${\tilde f}^\infty={\tilde g}^\infty+{\tilde T}_+,$$
which has the jump singularity at $x^*$.

To prove the approximation order, we use the fact that the subdivision scheme is a local procedure and that it reproduces cubic polynomials. Furthermore, by \cite{DD},  \cite{DLG}, for $f\in C^4(\mathbb{R})$ the 4-point subdivision scheme gives $O(h^4)$ approximation order to $f$. Therefore, away from the singularity, for $x<x_{j-2}$ and for $x>x_{j+3}$,
$$|\tilde{g}^\infty-(f-\tilde{T}_+)|=O(h^4),$$
as $h\to 0$. Correcting the approximation by $\tilde{T}_+$ yields the approximation result away from $x^*$.

To show the approximation order near $x^*$, we rewrite the data for the subdivision process, $\tilde{g}(x_j)=f(x_j)-\tilde{T}_+(x_j)$, as
$$\tilde{g}(x_j)=f(x_j)-T_+(x_j)-T_3^-(x_j)+(T_+(x_j)-\tilde{T}_+(x_j))+T_3^-(x_j).$$
Denoting $q=f-T_+-T_3^-$, we observe that $q=f-T_3^-$ for $x<x^*$ and
$q=f-T_3^+$ for $x>x^*$. In both cases $q(x)=O(h^4)$ near $x^*$. Next we note that by (\ref{order4})
$$(T_+(x_i)-\tilde{T}_+(x_i))=O(h^4),$$
for $i-4\le j\le i+5$.
Using the above, plus the locality of the subdivision scheme and the reproduction of cubic polynomials, it follows that
$$\tilde{g}^\infty(x)=T_3^-(x)+O(h^4), \ \ x_{j-2}\le x\le x_{j+3}.$$
Finally, in view of (\ref{T}) and (\ref{apprT}) we obtain
\begin{equation}\label{Fin}
\tilde{f}^\infty(x)=\tilde{g}^\infty(x)+\tilde{T}_+(x)=
\left\{\begin{array}{l}
T_3^-(x)+O(h^4),\ \ \ \ x_{j-2}\le x\le x^*,\\
T_3^+(x)+O(h^4),\ \ \ \ x^*\le x\le x_{j+3},
\end{array}\right.
\end{equation}
implying that
$$\tilde{f}^\infty(x)=f(x)+O(h^4), \ \ \ x_{j-2}\le x\le  x_{j+3},\ x\ne x^*,$$
which completes the proof.
\end{proof}

{
\begin{remark}
In order to include the point $x^*$ in Theorem \ref{Theorem1}, we need to suppose that the initial data comes from a function that is $C^4$ smooth from the left and from the right of $x^*$, as this information, as well as the exact position of the discontinuity is lost in the discretization process.
\end{remark}
}
Now we can continue with the case when we find corner singularities, i.e. discontinuities in the first derivative.

{
\begin{theorem}{\bf - The point-values' case: discontinuities in the first derivative.}\label{Theorem2}
Let $f$ be a piecewise $C^4$-smooth function on $[0, 1]$,
with a jump discontinuity in the first derivative, and assume we are given the
vector of values $\{f_j\}_{j=0}^{N}$ at the data points
$\{x_j = jh\}_{j=0}^{N}$, $N=1+1/h$. Let us suppose that the singular point is placed at $s^*\in (x_j,x_{j+1})$ and that we have located the singularity at the approximated location $x^*\in (x_j,x_{j+1})$. Then, the a\-ppro\-ximation obtained by the RC algorithm interpolates the data points, it is $C^{2-}\{[0,x^*)\cup (x^*,1]\}$, and $||f-\tilde{f}^\infty||_{\infty} = O(h^4)$ as $h\to 0$.
\end{theorem}
\begin{proof}
Outside the interval $(x^*, s^*)$ (if the approximated location $x^*$ is to the right of $s^*$, the analysis is equivalent), the proof is similar to the one followed in Theorem \ref{Theorem1}.

Let's now analyze the case when we subdivide in the interval $(x^*, s^*)$. We can just follow the proof of Theorem 1 in \cite{ACDD}. We suppose that the discontinuity is placed in the interval $[x_{j}, x_{j+1}]$, being $h=x_{j+1}-x_{j}$ a uniform grid-spacing. A graph of this case is plotted in Figure \ref{disc_centro}. As we are using cubic polynomials for the location of the discontinuity, it follows that the distance between $x^*$ and $s^*$ is $O(h^4)$, as proved in statement 3 of Lemma 3 of \cite{ACDD}. Let us analyze the accuracy attained by the subdivision scheme in the point-values discretization in the interval $(x^*, s^*)$. The left side of the discontinuity placed at $s^*$  is labeled as the $-$ side and the right side of $s^*$ as the $+$ side. The approximated location of the discontinuity is labeled as $x^*$. Following the process described in Section \ref{calc_disc} we obtain the jumps in the function and its derivatives at $x^*$. It is clear that if the piecewise function is not composed of polynomials of degree smaller or equal than 3, the location of the discontinuity will be approximated. In this case, obtaining the jump in the function and its derivatives at $x^*$ is just as extending $f^+$ from $s^*$ up to $x^*$, as shown in Figure \ref{disc_centro}.
Let's write the limit function obtained by the RC algorithm as,
\begin{equation}\label{e1}
\tilde{f}^\infty(x)=\left\{\begin{array}{ll}
(\tilde{f}^\infty)^-(x)&x<x^*,\\
(\tilde{f}^\infty)^+(x)&x>x^*.
\end{array}\right.
\end{equation}
The error obtained at any point $x\in (x^*, s^*)$ will be,
\begin{equation}\label{e1}
|f(x)-\tilde{f}^\infty(x)|=|f^-(x)-(\tilde{f}^\infty)^+(x)|\le|f^-(x)-f^+(x)|+ |f^+(x)-(\tilde{f}^\infty)^+(x)|.
\end{equation}
Following similar arguments to the ones used in the first part of the proof, the second term is bounded  and is $O(h^4)$. For the first term, we can use second order Taylor expansion around $s^*$ to write,
\begin{equation}\label{e2}
\begin{split}
|f^-(x)-f^+(x)|
\le |[f']|(s^*-x)+\sup_{\mathbb{R}\backslash s^*}|f''|(s^*-x)^2\\
\le (s^*-x^*)\left(|[f']|+\sup_{\mathbb{R}\backslash s^*}|f''|h\right).
\end{split}
\end{equation}
As $h<h_c$ (see (\ref{hc2})), then
\begin{equation}\label{e3}
\begin{aligned}
|f^-(x)-f^+(x)|&\le \frac{5}{4}|[f']|(s^*-x^*).
\end{aligned}
\end{equation}
Now, we can use point 3 of Lemma 2 of \cite{ACDD}, that establishes that $(s^*-x^*)\le C\frac{h^m\sup_{x\in\mathbb{R}\backslash\{s^*\}} |f^{(m)}(x)|}{|[f^{'}]|}$ (with $m=4$ in our case), to obtain that the first term of the right hand side of (\ref{e1}) is also bounded and is $O(h^4)$. This concludes the proof.
\end{proof}
}

\begin{figure}[!ht]
\begin{center}
\resizebox{13cm}{!} {
\begin{picture}(500,130)(0,0)
\put(50,-10){\line(1,0){470}}
\put(80,-10){\line(0,1){5}}
\put(160,-10){\line(0,1){5}}
\put(290,-10){\line(0,1){5}}
\put(262,-10){\line(0,1){5}}
\put(240,-10){\line(0,1){5}}
\put(320,-10){\line(0,1){5}}
\put(400,-10){\line(0,1){5}}
\put(480,-10){\line(0,1){5}}
\put(60,-20){$\cdots$}
\put(75,90){$f_{j-2}$}
\put(155,78){$f_{j-1}$}
\put(240,53){$f_{j}$}
\put(305,46){$f_{j+1}$}
\put(400,60){$f_{j+2}$}
\put(470,63){$f_{j+3}$}
\put(500,-20){$\cdots$}

\put(75,-23){$x_{j-2}$}
\put(155,-23){$x_{j-1}$}
\put(260,-3){$x^*$}
\put(290,-3){$s^*$}
\put(235,-23){$x_{j}$}
\put(315,-23){$x_{j+1}$}
\put(395,-23){$x_{j+2}$}
\put(475,-23){$x_{j+3}$}


\put(80,-10){\line(0,1){87}}
\put(80,77){\circle{5}}
\put(160,-10){\line(0,1){73}}
\put(160,63){\circle{5}}
\put(240,-10){\line(0,1){55}}
\put(240,45){\circle{5}}
\put(320,-10){\line(0,1){50}}
\put(320,40){\circle{5}}
\put(400,-10){\line(0,1){57}}
\put(400,47){\circle{5}}
\put(480,-10){\line(0,1){66}}
\put(480,56){\circle{5}}

\qbezier(70,79)(250,50)(290,35)
\qbezier(290,35)(330,45)(540,60)
\linethickness{0.5pt}
\qbezier[40](263,10)(275,30)(320,40)
\put(210,70){\circle{10}}
\put(206,68){$-$}
\put(350,70){\circle{10}}
\put(346,68){$+$}
\end{picture}
}
\end{center}
\caption{In this figure we can see an example of a corner singularity placed in the interval $(x_{j}, x_{j+1})$. The left side of the discontinuity placed at $s^*$  is labeled as the $-$ side and the right side of $s^*$ as the $+$ side. The approximated location of the discontinuity is labeled as $x^*$.}\label{disc_centro}
\end{figure}
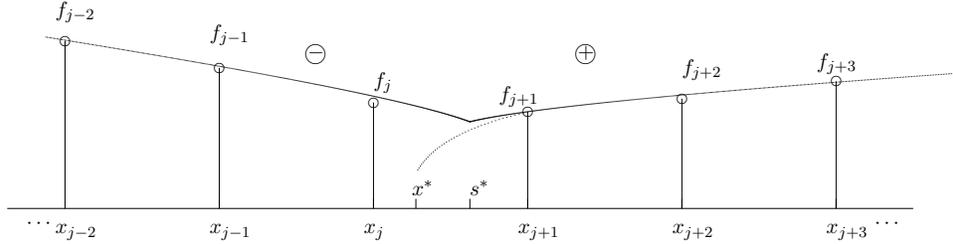

Due to the order of accuracy attained by the RC algorithm close to the discontinuity and to the regularity and convergence of the linear scheme we can state the following corollary:
\begin{corollary}
The RC scheme does not introduce Gibbs phenomenon nor diffusion close to the discontinuities.
\end{corollary}
\begin{proof}
Given the fact that the correction introduced in step 1 of the algorithm eliminates the presence of the discontinuity up to the order of accuracy of the linear scheme, we can consider that the linear scheme in step 2 is a\-pplied to smooth data after the correction, so no Gibbs effect can appear. Step 3 consists in correcting back the subdivided data, so no diffusion can appear.
\end{proof}
\begin{corollary}
The RC scheme reproduces piecewise cubic polynomials.
\end{corollary}
\begin{proof}
For piecewise cubic polynomials, the location of the discontinuity is exact. Remind that we are using cubic interpolating polynomials for the location algorithm. Following Theorem \ref{Theorem1} we get the proof. Mind that if the singularity is in the first derivative, the proof given for Theorem \ref{Theorem1} would be equivalent for this case, as now we know that the exact position of the singularity is $x^*$.
\end{proof}


In what follows, we enumerate the main properties of the RC algorithm:

\begin{itemize}
\item The RC scheme has the same (piecewise) regularity as the linear subdivision used.

\item The RC scheme does not produce any oscillations in the presence of singularities.

\item The algorithm has high accuracy without diffusion.

\item The algorithm is interpolatory. 

\item The RC scheme is exact for piecewise polynomial functions of the same degree as the one used in the
construction of the scheme.

\item The RC scheme is stable due to the stability of the linear scheme used.
\end{itemize}

Let us consider now the cell-averages discretization. In this case, we will only analyze the detection of jumps in the function.

\begin{theorem}{\bf - The cell-averages' case.}\label{teocell}
Let $f$ be a piecewise $C^3$-smooth function on $[0, 1]$, with a jump discontinuity at $s^*$, and assume we are given the cell-averages
$$\bar{f}_j=\frac{1}{h}\int^{x_j}_{x_{j-1}}f(x)dx,\quad j=1, \cdots, N.$$
Applying the algorithm in Section \ref{algorithm2} to the data sequence $\{F_j\}$ defined in (\ref{primitive}) to approximate the primitive function $F$, the following results hold:

1. The approximate singular point $x^*$ satisfies
$|s^*-x^*|\le Ch^4.$

2. The approximation $G(x)$ to the primitive function $F(x)$ interpolates the data $\{F_j\}_{j=1}^N$, and $|F(x)-G(x)|\le Ch^4$, $x\in [0,1]$.

3. The approximation $g(x)=G'(x)$ to $f(x)$ satisfies $|f(x)-g(x)|=O(h^3)$ as $h\to 0$ for $x\in [0,1]$, $x$ not in the closed interval between $s^*$ and $x^*$.

4. Denoting $\delta=x^*-s^*$,
$$|f(x)-g(x+\delta)|\le Ch^3\ \ \forall x\in [0,1]\setminus \{s^*\}.$$

\end{theorem}

\begin{proof}
Figure \ref{disc_centro2} represents a possible scenario for this theorem. Let us prove every point of the theorem:
\begin{enumerate}
\item First we note that the primitive function $F$ is piecewise $C^4$.
After locating the interval $[x_j,x_{j+1}]$ containing the singular point, the algorithm finds the two cubic polynomials, $\tilde{T}_-$ and $\tilde{T}_+$ , respectively interpolating the data at four points to the left and to the right of the singularity interval. As we are using cubic polynomials for the location of the singularity, the intersection point $x^*$ of the two polynomials in $[x_j,x_{j+1}]$ satisfies $|s^*-x^*|\le Ch^4$, { as proved in point 3 of Lemma 3 of \cite{ACDD}}.
\item Now we can follow the steps of Theorem \ref{Theorem1} to deduce  that $|F(x)-G(x)|\le Ch^4$, $x\in [0,1]$.
\item We also use here the result in \cite{Ruibin} that  for $F\in C^4(\mathbb{R})$ the 4-point subdivision scheme gives $O(h^3)$ approximation order to the derivative of $F$. This, together with the observation that both $F$ and $G$ are differentiable outside the interval between $s^*$ and $x^*$, yield the result in item 3.

\item To prove item 4 let us assume, w.l.o.g., that $x^*>s^*$. Both $f$ and $g(\cdot+\delta)$  have their jump discontinuity at $s^*$. For $x<s^*$ we have $x+\delta<x^*$.
Since $|\delta|=O(h^4)$, and $g$ is continuous for $x<x^*$, it follows that
$$g(x+\delta)=g(x)+O(h^4).$$
Hence,
$$|f(x)-g(x+\delta)|=|f(x)-g(x)|+O(h^4),$$
and by item 3, $|f(x)-g(x)|=O(h^3)$.
Similarly, for $x>s^*$, $x+\delta>x^*$ and $|f(x+\delta)-f(x)|=O(h^4)$. Hence,
$$|f(x)-g(x+\delta)|=|f(x+\delta)-g(x+\delta)|+O(h^4)=O(h^3).$$
\end{enumerate}
\end{proof}

\begin{figure}[!ht]
\begin{center}
\resizebox{13cm}{!} {
\begin{picture}(500,130)(0,0)
\put(50,-10){\line(1,0){150}}
\put(220,-10){\line(1,0){140}}
\put(80,-10){\line(0,1){5}}
\put(160,-10){\line(0,1){5}}
\put(290,-10){\line(0,1){5}}
\put(262,-10){\line(0,1){5}}
\put(102,-10){\line(0,1){5}}
\put(132,-10){\line(0,1){5}}
\put(240,-10){\line(0,1){5}}
\put(320,-10){\line(0,1){5}}
\put(50,-20){$\cdots$}
\put(180,-20){$\cdots$}
\put(215,-20){$\cdots$}
\put(340,-20){$\cdots$}
\put(240,53){$F_{j}$}
\put(305,46){$F_{j+1}$}

\put(70,-23){$x_{j}$}
\put(155,-23){$x_{j+1}$}
\put(100,-3){$s^*$}
\put(130,-3){$x^*$}
\put(102,-10){$\underbrace{\hspace{1.03cm}}_{\delta}$}
\put(260,-3){$s^*$}
\put(290,-3){$x^*$}
\put(262,-10){$\underbrace{\hspace{1cm}}_{\delta}$}

\put(235,-23){$x_{j}$}
\put(315,-23){$x_{j+1}$}


\put(240,-10){\line(0,1){55}}
\put(240,45){\circle{5}}
\put(320,-10){\line(0,1){50}}
\put(320,40){\circle{5}}

\linethickness{0.5pt}
\qbezier(60,70)(100,50)(132,55)
\qbezier(132,25)(150,16)(190,15)
\qbezier[15](102,45)(120,30)(132,25)

\qbezier(230,48)(250,43)(290,35)
\qbezier[40](240,44)(250,42)(263,10)
\qbezier(290,35)(320,40)(330,43)
\linethickness{0.5pt}
\qbezier[40](263,10)(275,18)(330,43)

\put(100,18){$f(x)$}
\put(100,63){$g(x)$}

\put(270,8){$F(x)$}
\put(270,43){$G(x)$}
\end{picture}
}
\end{center}
\caption{In this figure we represent a jump discontinuity in the interval $(x_{j}, x_{j+1})$, that transforms into a discontinuity in the first derivative for the primitive. The approximated location of the discontinuity is labeled as $x^*$ and the exact location is labeled as $s^*$. The exact function is represented by $f(x)$ and the approximation is represented by $g(x)$. The exact primitive is represented by $F(x)$ and the approximation is represented by $G(x)$. }\label{disc_centro2}
\end{figure}
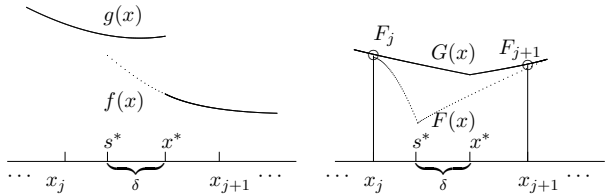

\begin{remark}
Inside the interval between $s^*$ and $x^*$, as mentioned in Section 7 of \cite{ACDD}, we can not hope to obtain a better accuracy than $O(1)$ in the infinity norm. Even though, we can provide a result in the $L^1$ norm (in fact, this result is provided for the $L^p$ norm in Section 7 of \cite{ACDD}). We use the inequality
$$||f^{\infty}- f||_{L^1}\leq ||f^{\infty} - g||_{L^1}+||g- f||_{L^1},$$
where $f^{\infty}$ is the limit function in the cell-averages and $g$ is the smooth function obtained by our algorithm plus the final translation, which has the same cell-averages as the function $f$ and has the discontinuity at $x^*$. The first term of the right hand side of the inequality is controlled by the order of the linear subdivision scheme (as the correction do not perturb the order) and the second term is controlled using the results in Section 7 of \cite{ACDD}.
\end{remark}

\section{Numerical results for the case of point-values sampling}\label{numexp_point}

In the experiments presented in this section we have calculated a high accuracy approximation of the position of the discontinuity using the technique described in the introduction. We have also computed high order approximations of the jumps in the function and its derivatives using the process described in Section \ref{calc_disc}.

Let's start by the function,
\begin{equation}\label{exp1}
f(x)=\left\{\begin{array}{ll}
a+\left(x-\frac{\pi}{6}\right)\left(x-\frac{\pi}{6}-10\right)+x^2+\sin(10x), & \textrm{if } x< \frac{\pi}{6}\\
x^2+\sin(10x),& \textrm{if } x\ge \frac{\pi}{6},
\end{array}\right.
\end{equation}
$x\in[0, 1]$. 

It is easy to check that, for $a=0$, the jump in the function at $x=\frac{\pi}{6}$ is $[f(x=\pi/6)]=0$, the jump in the first derivative is $[f'(x=\pi/6)]=10$ and in the second derivative is $[f''(x=\pi/6)]=-2$.

We compare the performance of three algorithms: One is the linear algorithm which is just the application of the 4-points subdivision to the data. The second is the quasi-linear ENO-SR discussed in \cite{ACDD} which uses special subdivision rules near the singularity. The third is the regularization-correction algorithm suggested here, which is called in short the RC or the corrected algorithm.
In Figure \ref{exp_point} we present the result obtained by the linear algorithm (left), the quasi-linear ENO-SR scheme (center) and the RC algorithm (right). In Figure \ref{exp_point_zoom} we present a zoom around the singularity. In order to obtain these graphs we have started from 16 data points of the original function and we have performed 5 levels of subdivision. We represent the original function with a dotted line, the discretized data at the lower resolution (16 data points) with blank circles and the limit function with a dashed line. We can see that the results obtained by the quasi-linear algorithm and the RC algorithm are very similar. Even though, as we will see in Subsection \ref{regupoint}, the regularity of the limit function obtained by the quasi-linear algorithm is lower than the one obtained by the RC approach.

Now we apply the algorithm to a function with more than one discontinuity. For example, the function
\begin{equation}\label{exp2}
f(x)=\left\{\begin{array}{ll}
\left(x-\frac{\pi}{12}\right)\left(x-\frac{\pi}{12}-10\right)+x^2+\sin(10x), & \textrm{if } x< \frac{\pi}{6}\\
x^2+\sin(10x),& \textrm{if } \frac{\pi}{12}\le x< \frac{3\pi}{12},\\
(x-\frac{3\pi}{12})(x-\frac{3\pi}{12}-5)+x^2+\sin(10x),& \textrm{if }  x\ge\frac{3\pi}{12},
\end{array}\right.
\end{equation}
that presents two discontinuities. In Figure \ref{exp_point2} we present the result obtained when applying the RC algorithm to the function in (\ref{exp2}). We can see that the algorithm can deal with more than one discontinuity with no problem, just applying steps 1 and 3 of Section \ref{nuevo} as many times as discontinuities are detected.

\begin{figure}[!ht]
\centerline{\psfig{figure=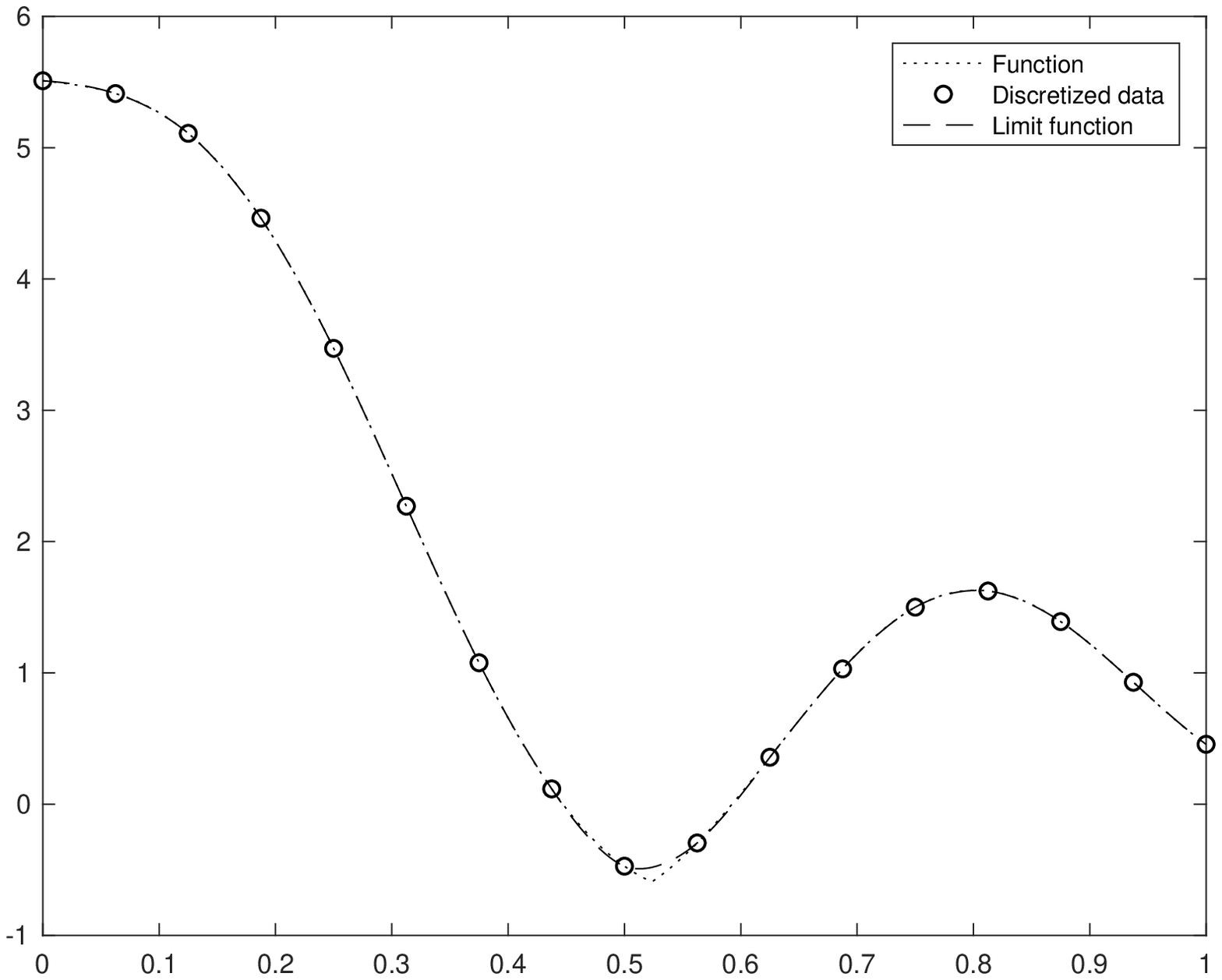,height=4cm}\\
\psfig{figure=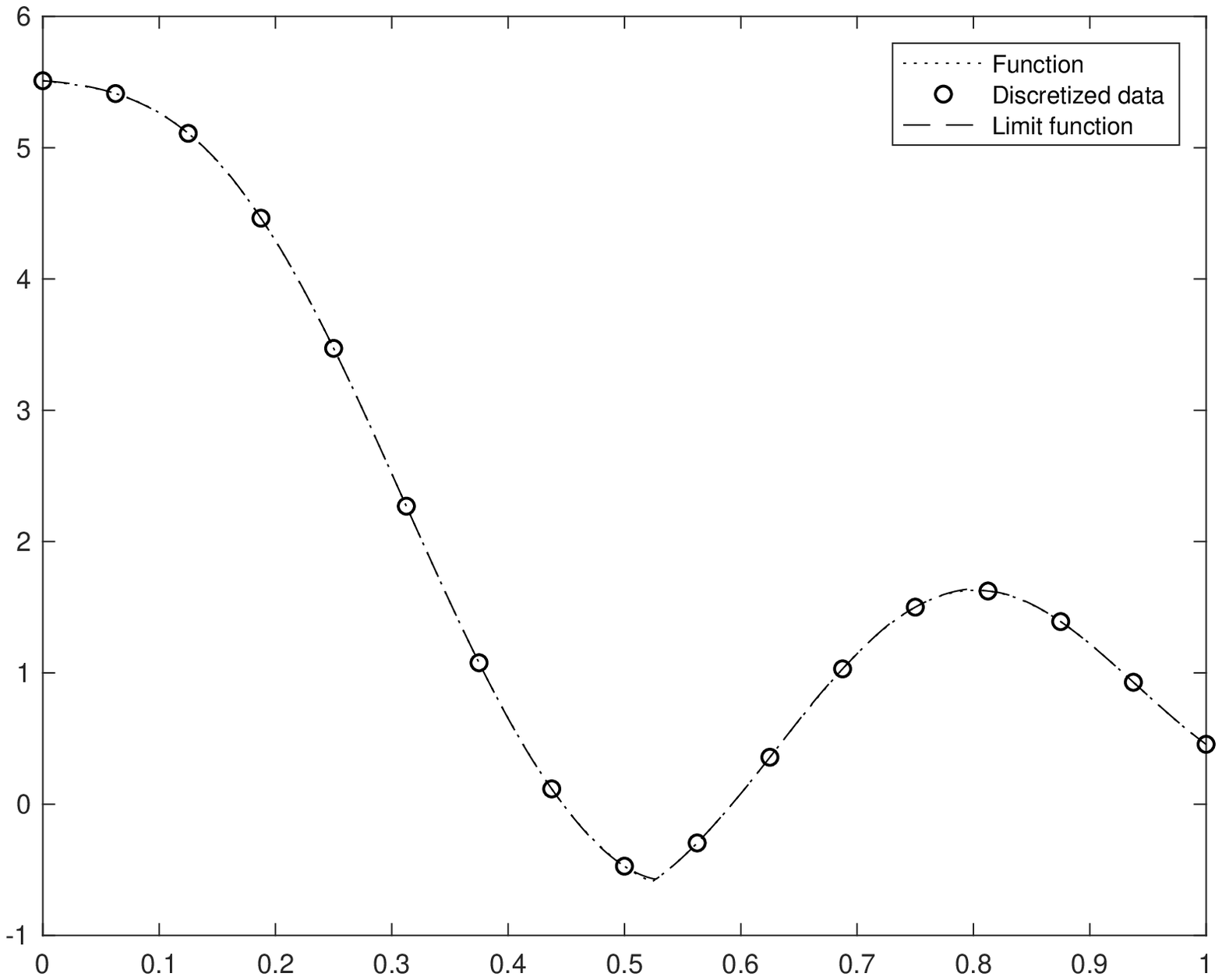,height=4cm}\\
\psfig{figure=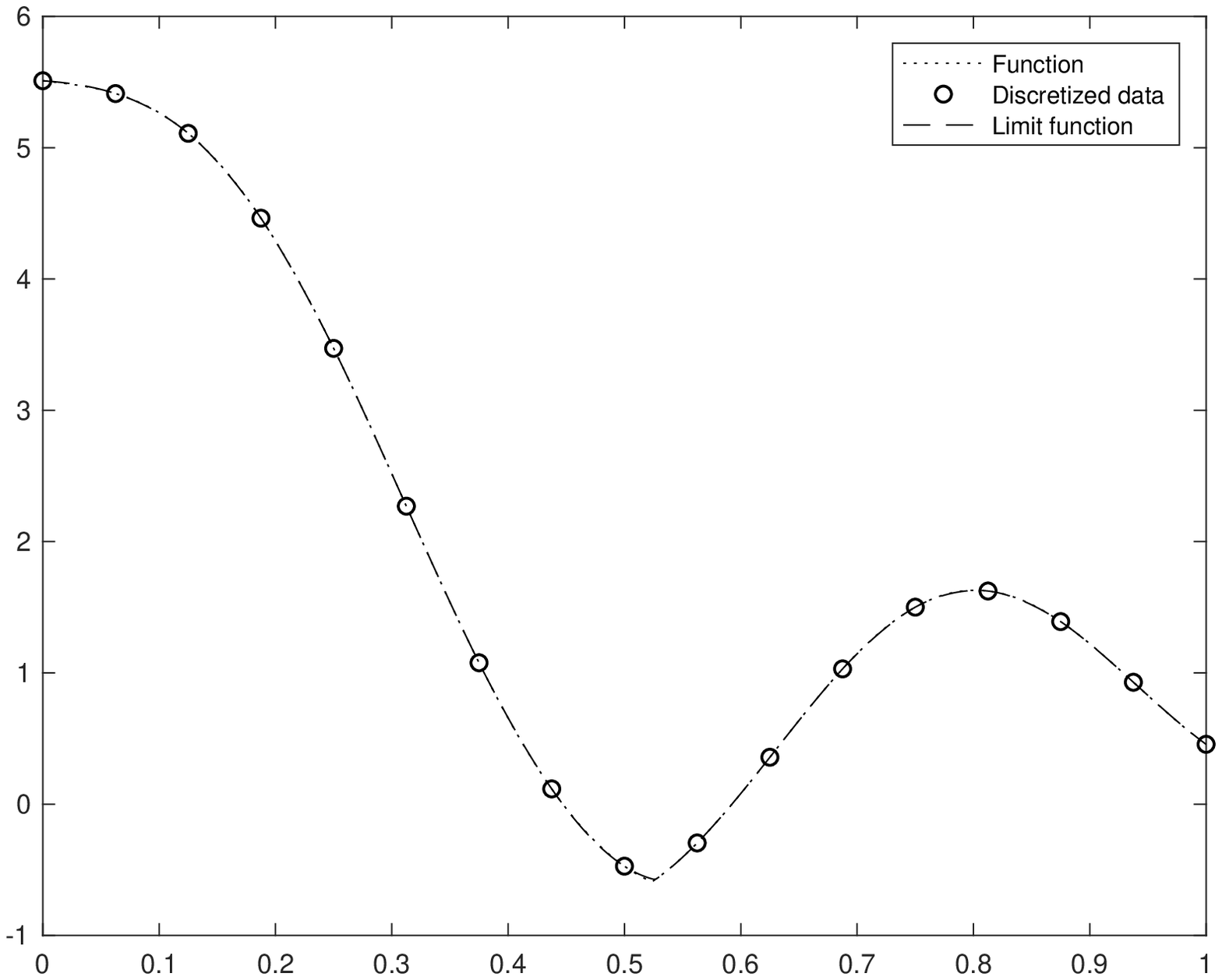,height=4cm}}
\caption{Limit function obtained by the linear algorithm (left), the quasi-linear algorithm (center) and the RC algorithm (right). In order to obtain these graphs we have started from 16 initial data points of the function in (\ref{exp1}).}\label{exp_point}
\end{figure}

\begin{figure}[!ht]
\centerline{\psfig{figure=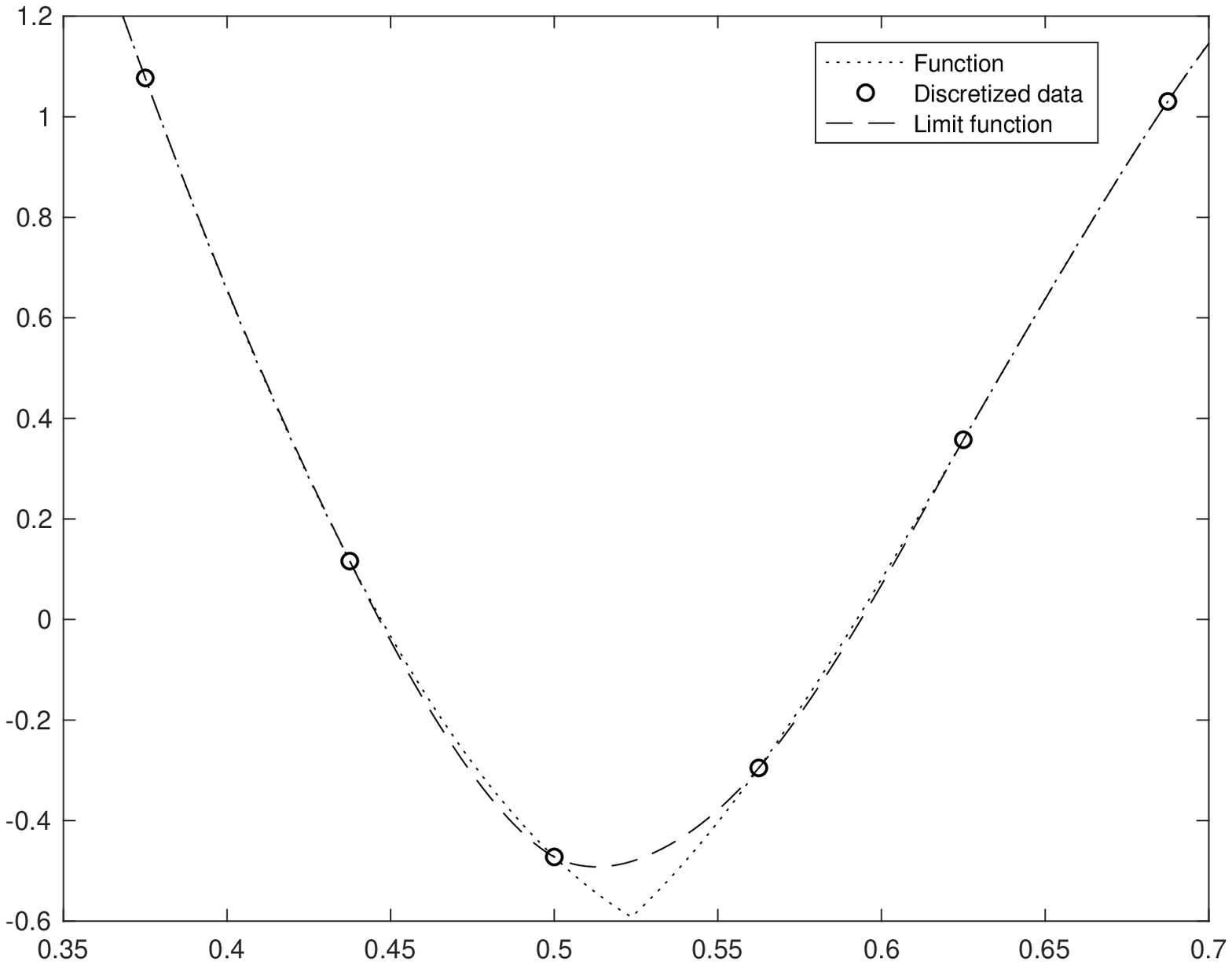,height=4cm}\\
\psfig{figure=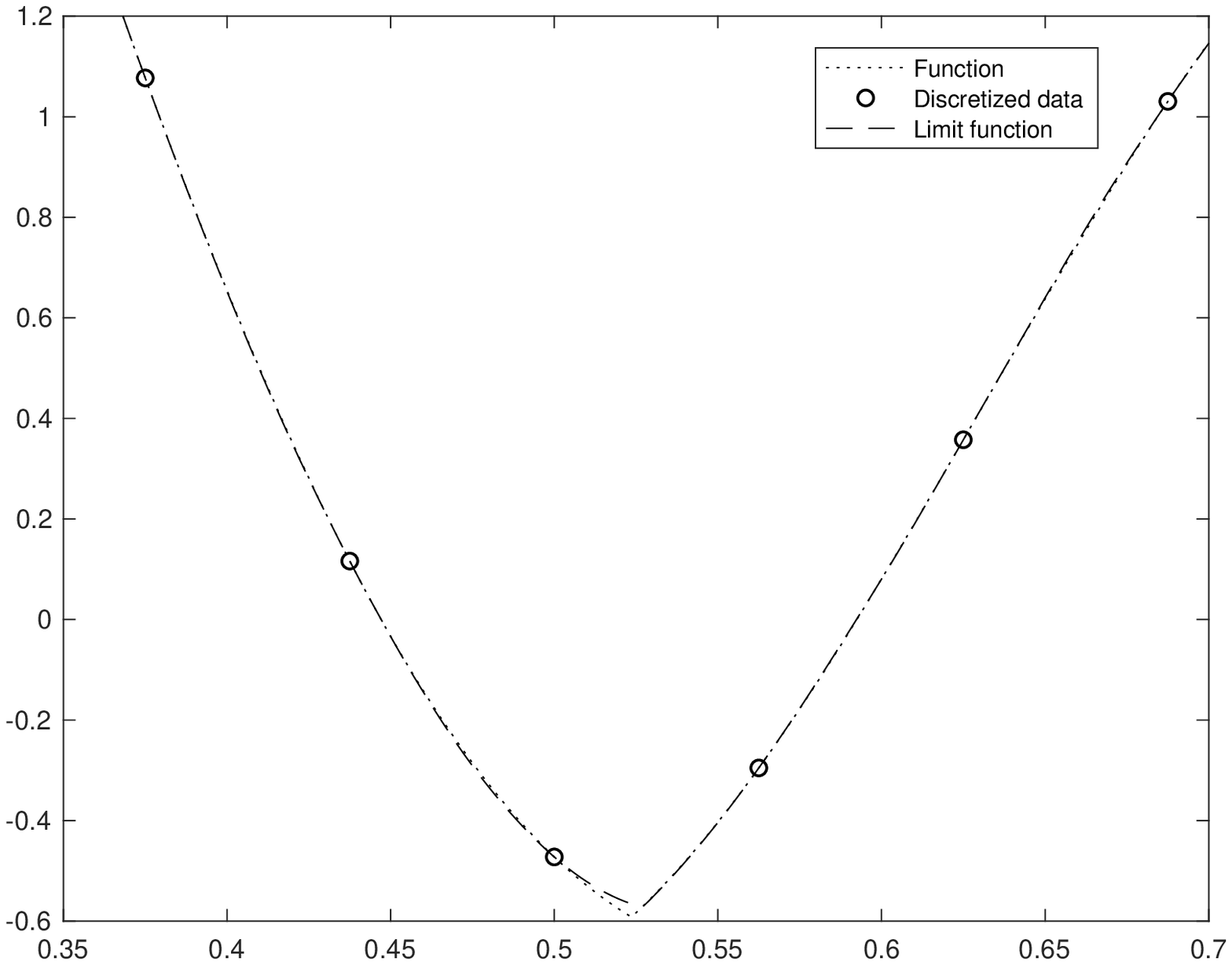,height=4cm}\\
\psfig{figure=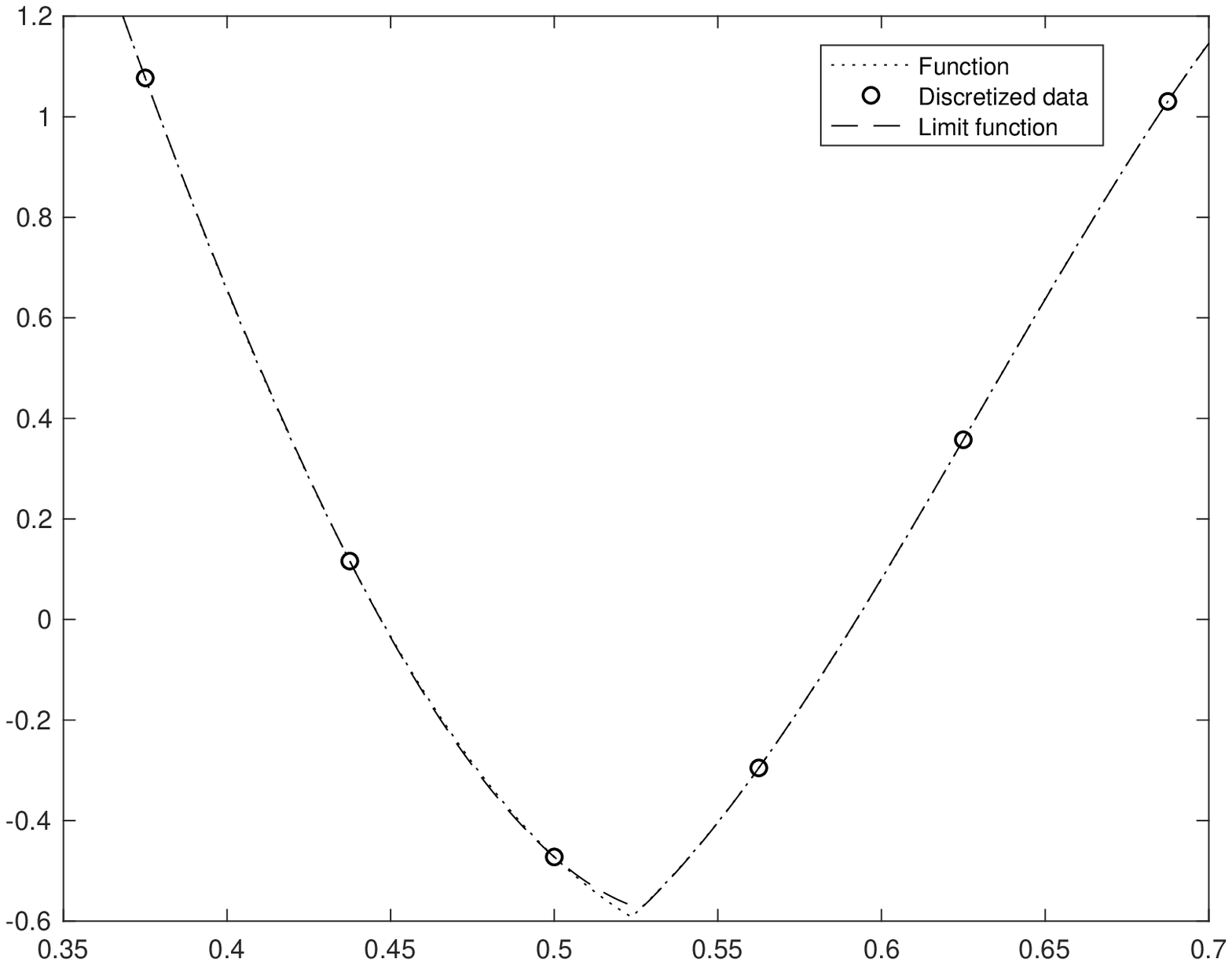,height=4cm}}
\caption{Zoom of the limit functions shown in Figure \ref{exp_point} obtained by the linear algorithm (left), the quasi-linear algorithm (center) and the RC algorithm (right).}\label{exp_point_zoom}
\end{figure}

\begin{figure}[!ht]
\centerline{\psfig{figure=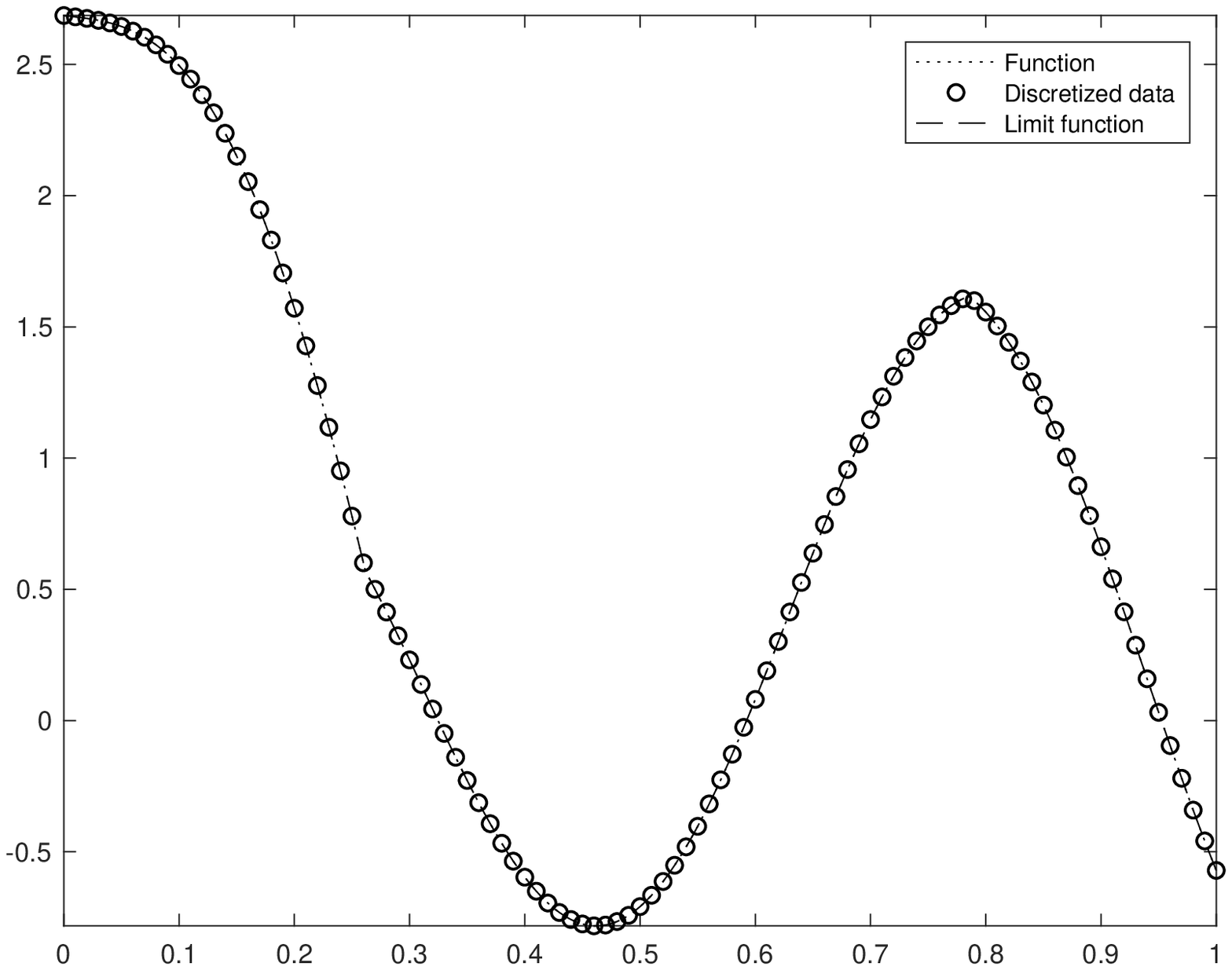,height=4cm}\\
\psfig{figure=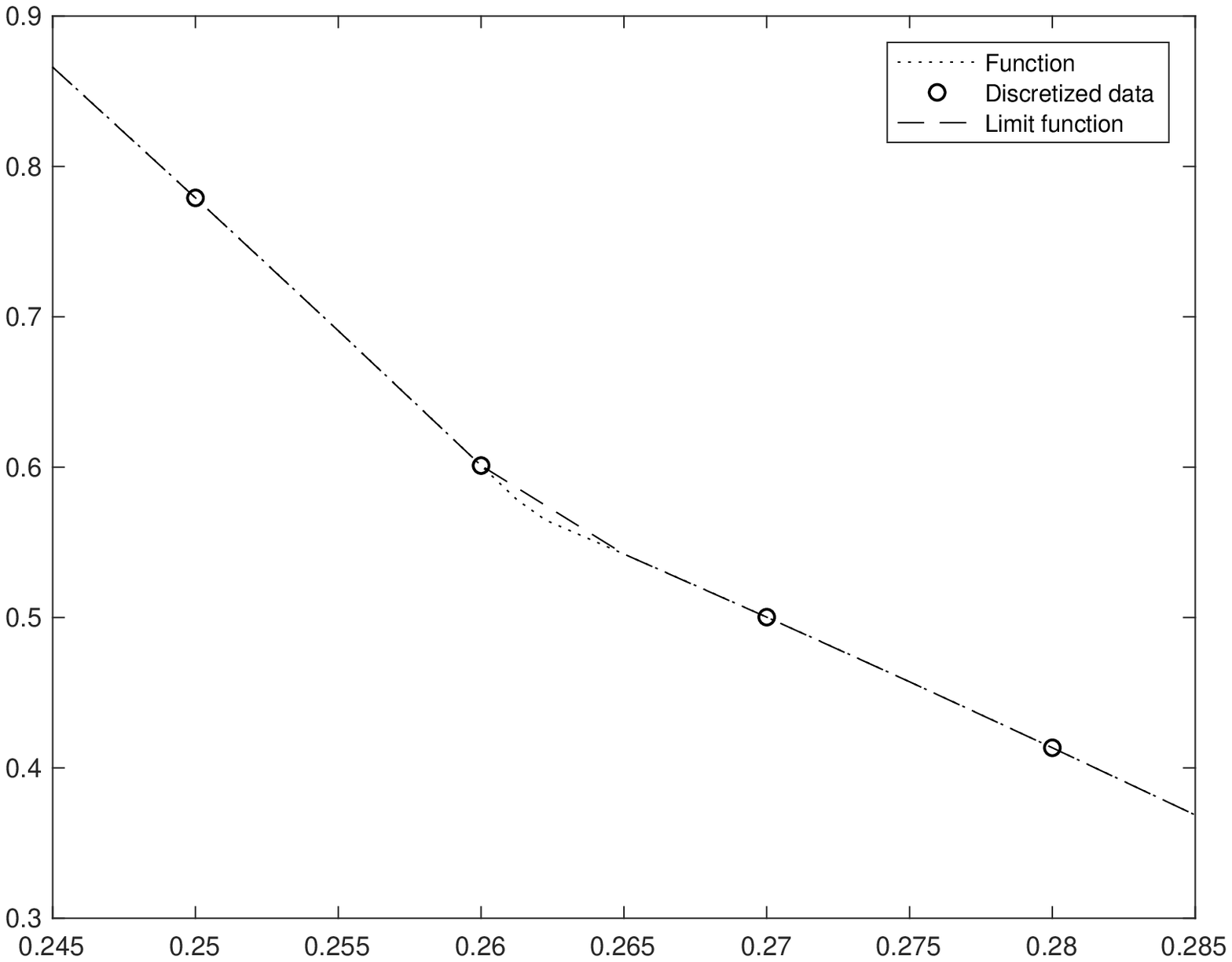,height=4cm}\\
\psfig{figure=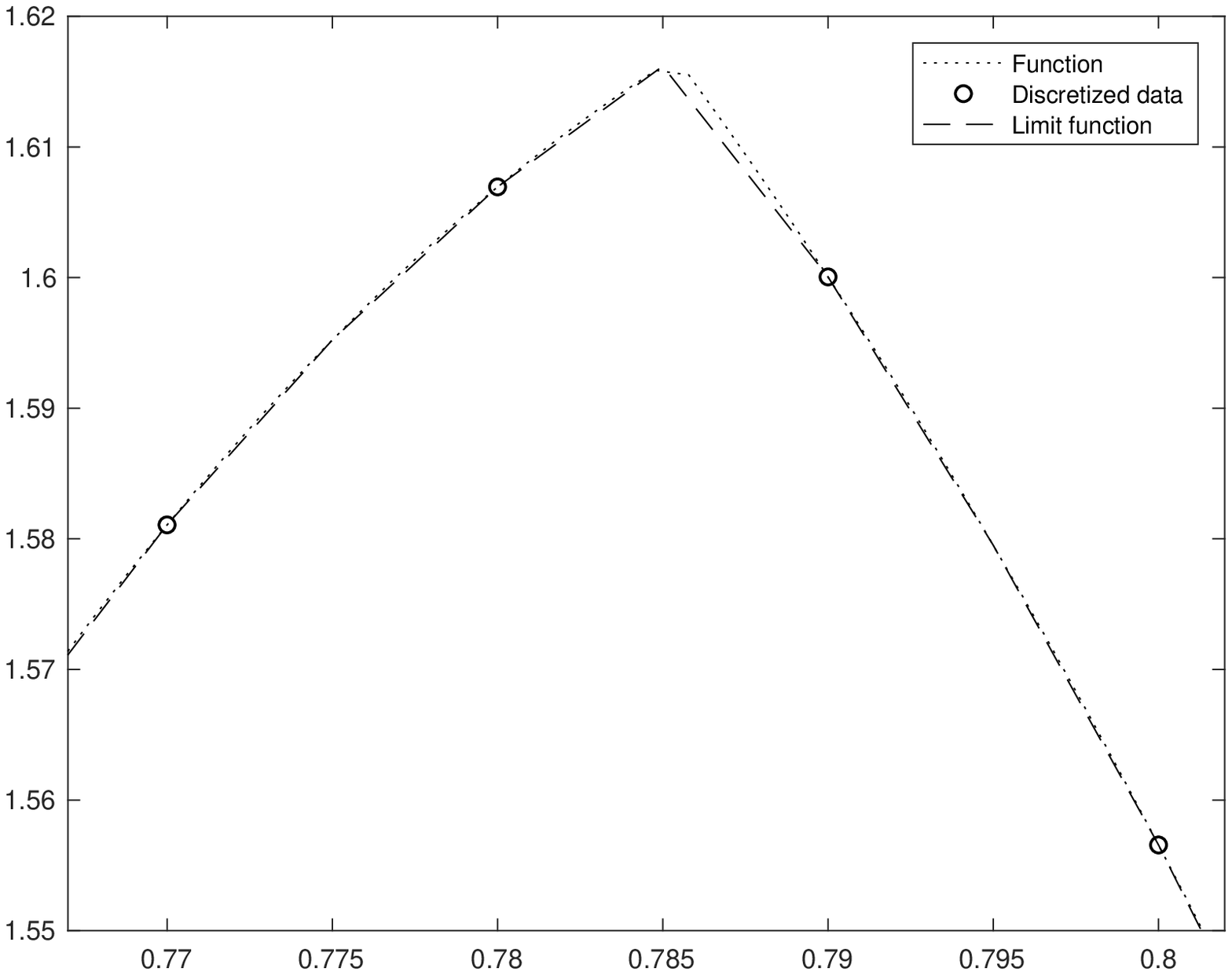,height=4cm}}
\caption{Limit function obtained by the RC algorithm after five levels of subdivision using as initial data the piecewise smooth function in (\ref{exp2}) that presents two singularities.}\label{exp_point2}
\end{figure}

\subsection{Jump discontinuities in the point-values sampling case}\label{section_jumps}

We have previously mentioned that it is not possible to locate the position of a discontinuity in the function using a discretization by point-values. Although this is true, it is indeed possible to locate the interval of length $h$ that contains the discontinuity. Our objective is to obtain subdivided data that keeps the regularity of the linear subdivision scheme applied, that does not present diffusion nor Gibbs phenomenon and that have optimal accuracy close to the discontinuities. Then, as argued in Theorem \ref{Theorem1}, we just assume that the discontinuity is placed at the middle of the interval. Let's apply the RC algorithm to data obtained from the sampling of the function,
\begin{equation}\label{exp4}
f(x)=\left\{\begin{array}{ll}
\left(x-\frac{\pi}{12}\right)\left(x-\frac{\pi}{12}-10\right)+x^2+\sin(10x)+1, & \textrm{if } x< \frac{\pi}{6},\\
x^2+\sin(10x),& \textrm{if } \frac{\pi}{12}\le x< \frac{3\pi}{12},\\
(x-\frac{3\pi}{12})(x-\frac{3\pi}{12}-5)+x^2+\sin(10x)+2,& \textrm{if }  x\ge\frac{3\pi}{12}.
\end{array}\right.
\end{equation}
Figure \ref{exp_point_zoom2} presents the result obtained after five levels of subdivision using $100$ initial points. As expected, we do not obtain Gibbs phenomenon nor diffusion. It is important to mention that in Figure \ref{exp_point_zoom2} we have plotted the limit function and the original data with continuous lines to point out the position of the jump in the function, but no data or subdivided data is placed in the middle of the jump.

In the following subsections we will check the regularity and the accuracy obtained using the RC algorithm.

\begin{figure}[!ht]
\centerline{\psfig{figure=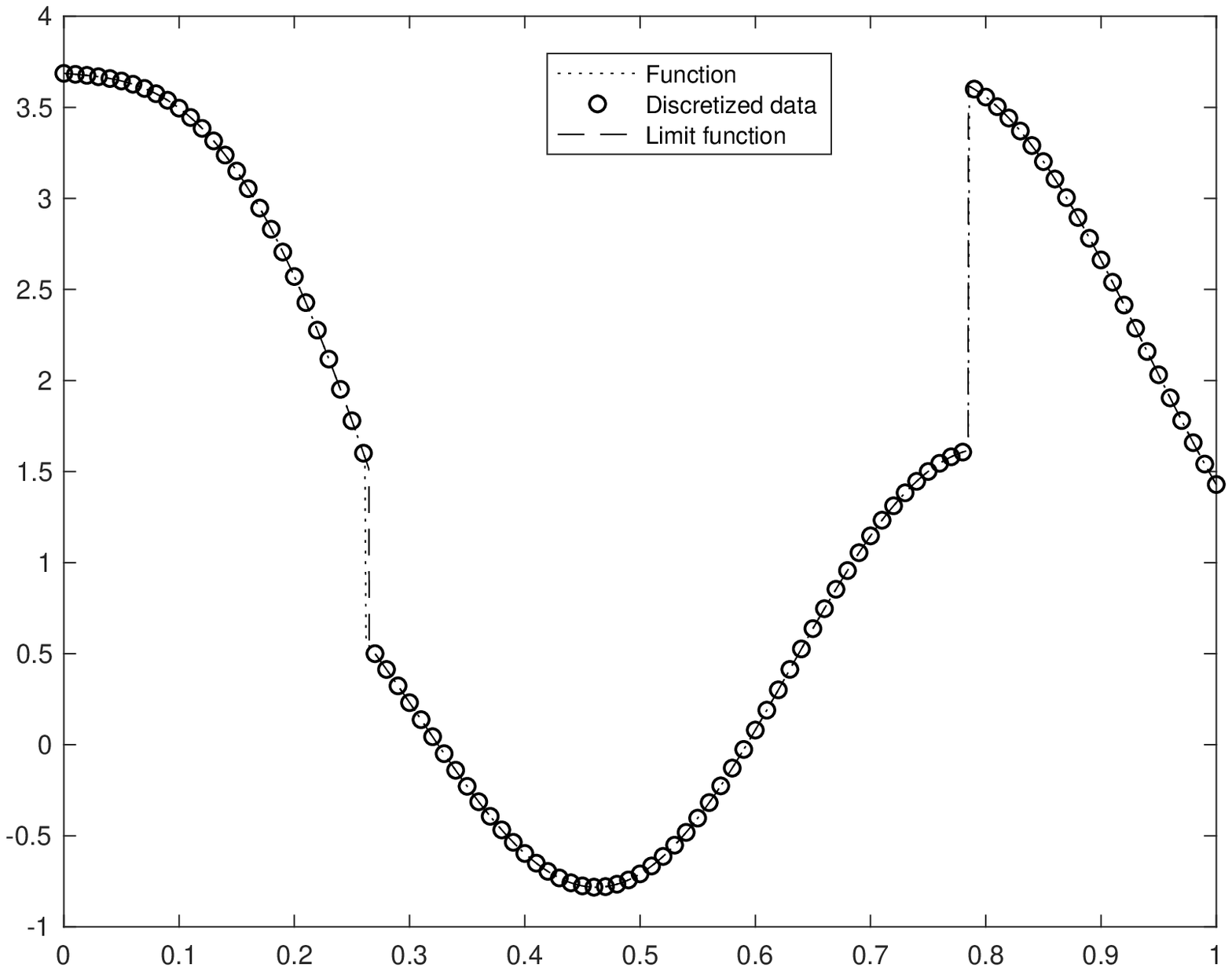,height=4cm}\\
\psfig{figure=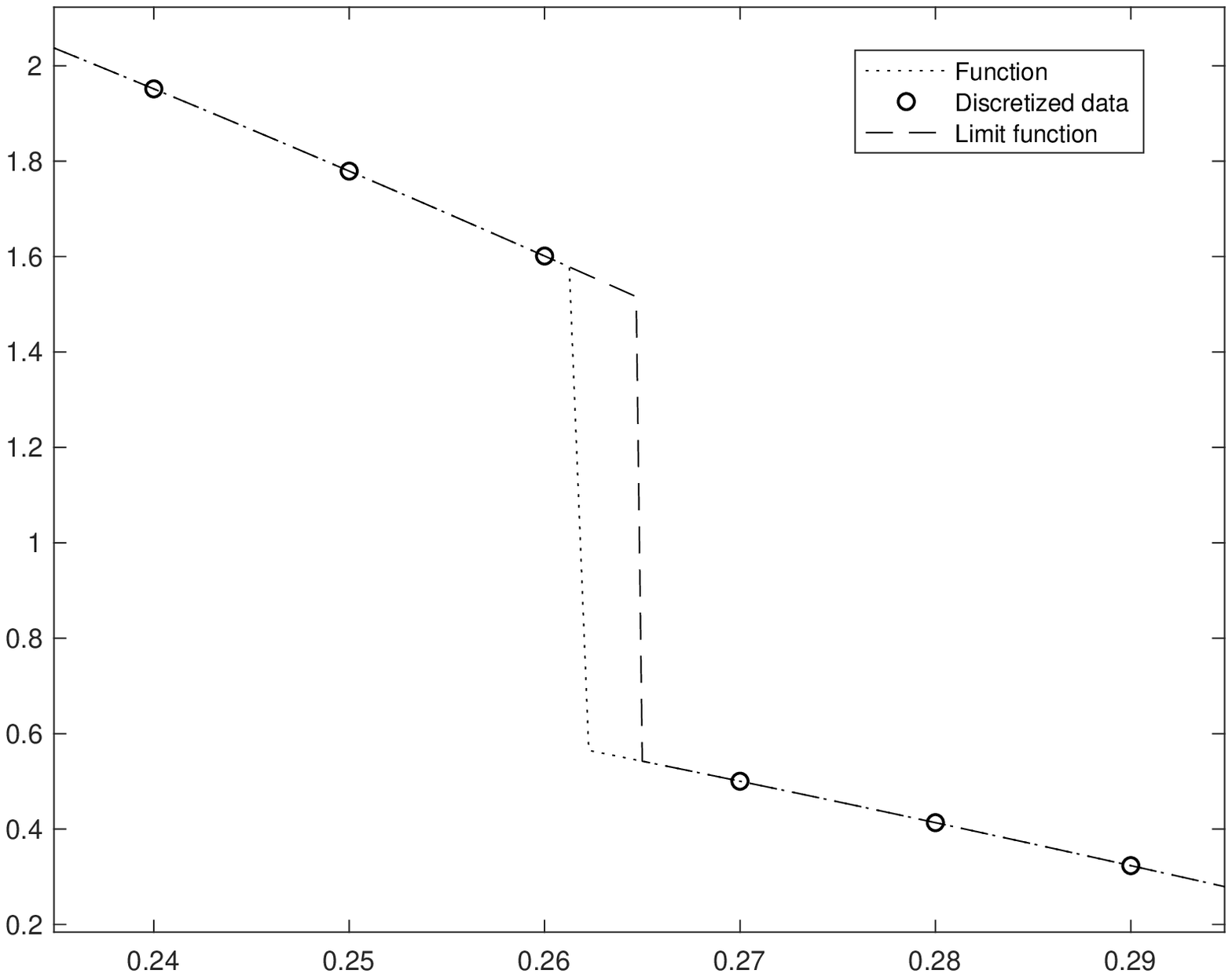,height=4cm}\\
\psfig{figure=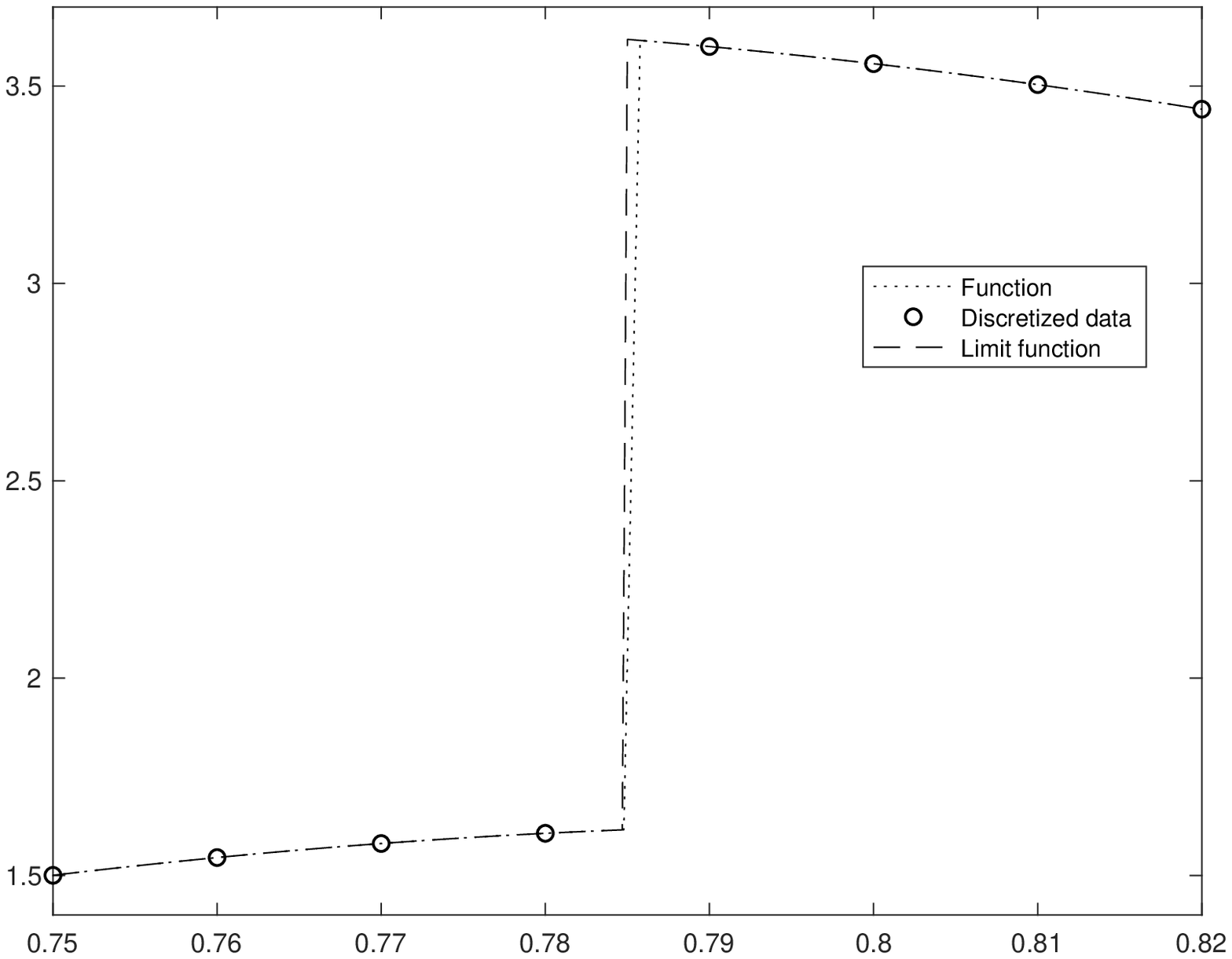,height=4cm}}
\caption{Limit function obtained after 5 levels of subdivision using the RC algorithm for a piecewise continuous function. The initial data has 100 points.}\label{exp_point_zoom2}
\end{figure}

\subsection{Numerical Regularity}\label{regupoint}


Following \cite{Kuijt98},  the regularity of a limit function of a subdivision process can be evaluated  numerically using the values $\{f^L_n\}$ obtained at subdivision level $j$ as
$$\beta_k=-\log_2\left(2^k\frac{||\Delta^{k+1}f^{L+1}_n ||_\infty}
{||\Delta^{k+1}f^{L}_n||_\infty} \right ).$$ This expression provides an estimate for $\beta_1$ and $\beta_2$ such that the limit functions belong to $C^{1+\beta_1-}$ and $C^{2+\beta_2-}$.

Let's consider the regularity of the limit function obtained when subdividing data acquired through a point-values discretization of the function (\ref{exp1}). In Table \ref{table_exp1} we present numerical estimations of the regularity constant of the linear algorithm, the quasi-linear algorithm and the RC algorithm. To obtain this table, we start from $100$ initial data points and we subdivide from $L=5$ to $L=10$ levels of subdivision in order to obtain an approximation of the limit function. We measure the numerical regularity for $x<\frac{\pi}{6}$ assuring that the singularity is not contained in the data. From this table we can see that the numerical estimate of the re\-gularity for the RC algorithm is  very close to the one obtained by
the linear scheme. The quasi-linear scheme clearly is less regular.

\begin{table}[ht!]

\begin{center}
\resizebox{13cm}{!}{
\begin{tabular}{|c|c|c|c|c|c|c|c|c|}
\hline
&$L$ &  5 & 6&  7   & 8  & 9 & 10 \\
\hline
\multirow{ 3}{*}{$\beta_1$}
&Linear& 0.8117    &  0.8335  &   0.8507  &   0.8647   &  0.8763  &   {\bf 0.8861} \\\cline{2-8}
&Quasi-linear& -6.0754   &    -0.9273   &   0.03565   &    0.0723    &   0.1556    &   {\bf 0.3724}\\\cline{2-8}
&RC&0.9967  &    0.9983  &    0.9992  &    0.9996  &    0.9998 &     {\bf  0.9999}  \\\cline{2-8}\hline
 \multirow{ 3}{*}{$\beta_2$}
&Linear& 0.0167  & 0.0084  &  0.0042  & 0.0021  &  0.0011 & {\bf 0.0005} \\\cline{2-8}
&Quasi-linear& -9.5160  &     -1.9331   &   -0.9654   &   -0.9281  &    -0.8446   &  {\bf  -0.6277} \\\cline{2-8}
&RC& 0.5414  &    0.2706 &     0.1156 &    0.0491  &   0.0227  &   {\bf 0.0103} \\\cline{2-8}
\hline
\end{tabular}
}
\end{center}
\caption{ Numerical estimation of the limit functions regularity $C^{1+\beta_1-}$ and $C^{2+\beta_2-}$ for the different schemes presented and the function in (\ref{exp1}). }
 \label{table_exp1}
\end{table}

\subsection{Grid refinement analysis in the point-values sampling case}

In this subsection we present an experiment oriented to check the order of accuracy of the schemes presented. In order to do this, we check the error of interpolation in the infinity norm obtained in the whole domain and then we perform a grid refinement analysis. We define the order of accuracy of the reconstruction as,
$$order_{k+1}=log_2\left(\frac{E^k_\infty}{E^{k+1}_\infty}\right),$$
$E_\infty^k$ being the $\ell_\infty$ error obtained using data with a grid spacing $h_k=N_k^{-1}$ and $E_\infty^{k+1}$ the error obtained with a grid spacing $h_k/2$.


For point-values data at the poins $\{x_j=jN_k^{-1}\}_{j=0, \cdots, N_k}$, we estimate $E^k_\infty$ using the values of the test function and its approximation on a mesh refined by a factor of $2^{-10}$.
Table \ref{precision_point} presents the results obtained by the three algorithms for the test function in (\ref{exp1}) with $a=0$. We can see how the linear algorithm losses the accuracy due to the presence of the discontinuity, while the quasi-linear algorithm and the RC approach keep high order of accuracy in the whole domain.

\begin{table}[!ht]
\begin{center}
\resizebox{10cm}{!} {
\begin{tabular}{|c|c|c|c|c|c|c|c|c|c|c|c|c|c|c|}
\hline
&\multicolumn{2}{|c|}{RC algorithm}&\multicolumn{2}{|c|}{Linear}&\multicolumn{2}{|c|}{Quasi-linear}\\
\hline $N_k$ &$E^k_{\infty}$ & $order_k$&$E^k_{\infty}$ & $order_k$&$E^k_{\infty}$ & $order_k$
            \\
\hline  16&   2.3041e-02  &-&1.1052e-01 &-&2.6140e-02& -
            \\
\hline  32&     5.3611e-03  &2.1036 & 4.0630e-02  &1.4437  & 5.3611e-03&2.2856
            \\
\hline  64&       1.6162e-04  &   5.0518    &2.9258e-02  &0.4737 &  1.6164e-04  &5.0516
            \\
\hline  128&     2.7694e-05   & 2.5450   & 5.2048e-03 & 2.4909& 2.7694e-05&2.5452
            \\
\hline  256&     1.7574e-06   & 3.9780   & 2.5469e-03  & 1.0311 &  1.7574e-06  & 3.9780
            \\
\hline  512&    1.0309e-07   & 4.0916  & 1.2169e-03&  1.0655& 1.0309e-07  &4.0916
            \\
\hline  1024 &   5.3956e-09   & 4.2559   & 8.9145e-04  & 0.4490 &5.3966e-09  &4.2557
            \\
\hline  2048 &    2.2313e-10   & {\bf 4.5958}  &7.8471e-04  & {\bf 0.1840} &2.2408e-10  &{\bf 4.5899}
            \\
\hline
\end{tabular}
}
\caption{Grid refinement analysis in the $l^{\infty}$ norm for the function in (\ref{exp1}) for the three algorithms.}\label{precision_point}
\end{center}
\end{table}

It is important to remark that for any initial data obtained from a piecewise polynomial function of degree smaller or equal than three, the RC scheme and the quasi-linear method achieve exact approximations within machine precision.

We also check the accuracy attained by the RC scheme when working with piecewise continuous functions as the ones analyzed in Subsection \ref{section_jumps}. For this experiment, the data at the resolution $k$ has been obtained through the sampling of the function in (\ref{exp1}) with $a=10$, that is a piecewise conti\-nuous function with a jump discontinuity of size $a$ at $x=\frac{\pi}{6}$. The philosophy of this experiment is the one explained in Subsection \ref{section_jumps}: as the exact position of the discontinuity is lost when discretizing a function by point-values, we consider that the discontinuity is placed at the middle of the suspicious intervals. Then, in order to obtain the error and the order of accuracy, we compare with the original function but with the discontinuity placed in the middle of the suspicious interval. The results are presented in Table \ref{precision_point_salto} and we can see that the RC algorithm attains the maximum possible accuracy in the infinity norm.


\begin{table}[!ht]
\begin{center}
\resizebox{4cm}{!} {
\begin{tabular}{|c|c|c|c|c|c|c|c|c|c|c|c|c|c|c|}
\hline
&\multicolumn{2}{|c|}{RC algorithm}\\
\hline $N_k$ &$E^k_{\infty}$ & $order_k$
            \\
\hline  16&3.6320e-02& -
            \\
\hline  32& 2.5607e-03&3.8262
            \\
\hline  64&1.5596e-04& 4.0373
            \\
\hline  128&  9.1954e-06 & 4.0841
            \\
\hline  256&  5.6303e-07 & 4.0296
            \\
\hline  512&  3.4794e-08& 4.0163
            \\
\hline  1024&   2.1618e-09 & 4.0085
            \\
\hline  2048&  1.3470e-10&{\bf 4.0044 }
            \\
\hline
\end{tabular}
}
\caption{Grid refinement analysis in the $l^{\infty}$ norm for the RC scheme for the function in (\ref{exp1}), that in this case is a piecewise continuous function with a jump discontinuity in the function equal to $a=10$. In this case, as the exact position of the discontinuity is lost when discretizing a function through the point-values, we consider that the discontinuity of the function in (\ref{exp1}) is not placed at $\pi/6$ but at the middle of the suspicious interval.}\label{precision_point_salto}
\end{center}
\end{table}

\subsection{RC approximation of bivariate point-values  data}
Taking into account what has been explained in previous subsections, we can try to approximate piecewise smooth two dimensional functions. Let's consider for example the next bivariate function,
\begin{equation}\label{exp2D_point}
f(x)=\left\{\begin{array}{ll}
\cos(\pi x)\cos(\pi y), & \textrm{if } (x+\frac{1}{2})^2+(y-\frac{1}{2})^2< 1\\
1-\cos(\pi x)\sin(\pi y),& \textrm{if } (x+\frac{1}{2})^2+(y-\frac{1}{2})^2\ge 1,
\end{array}\right.
\end{equation}
that is displayed in Figure \ref{2D1}. 
The additional challenge here is locating and to approximating the singularity curve.
Assuming that the singularity curve is nowhere parallel to the $x$-axis, we can apply here the level set function approach presented in \cite{LevinFourier}.

We can proceed as follows:
\begin{itemize}
\item Detect and locate the possible discontinuities using the rows data. Use this information to build an approximate signed-distance data from the unknown singularity curve. Now fit a spline surface $S(x,y)$ to this data, and construct the zero level set of this surface. This approach from \cite{LevinFourier} gives an $O(h)$ approximation to the singularity curve.
\item Obtain and store the one dimensional correction terms for all the rows.
\item Add the correction term to the rows and apply to the corrected data the tensor product linear 4-point subdivision algorithm to generate the first stage approximation $g(x,y)$.
\item Subdivide along columns the correction terms to obtain a smooth function representing the correction term $T(x,y)$ over the whole domain.
\item Use the level set function $S(x,y)$ to set $T(x,y)$ to zero to the right of the approximate singularity curve.
\item Subtract the masked correction term from $g(x,y)$.
\end{itemize}
Of course, in order to apply this technique successfully, it is necessary that discontinuities are far enough from each other and from the boundaries. Figure \ref{2D1} top to the left presents the original bivariate data sampled from the function in (\ref{exp2D_point}). Top to the right we present the resultant subdivided data obtained following the process described before. Bottom left we can observe the subdivided correction term, that is clearly smooth. Bottom to the right we can see the subdivided and masked correction term.
\begin{figure}[!ht]
\centerline{\psfig{figure=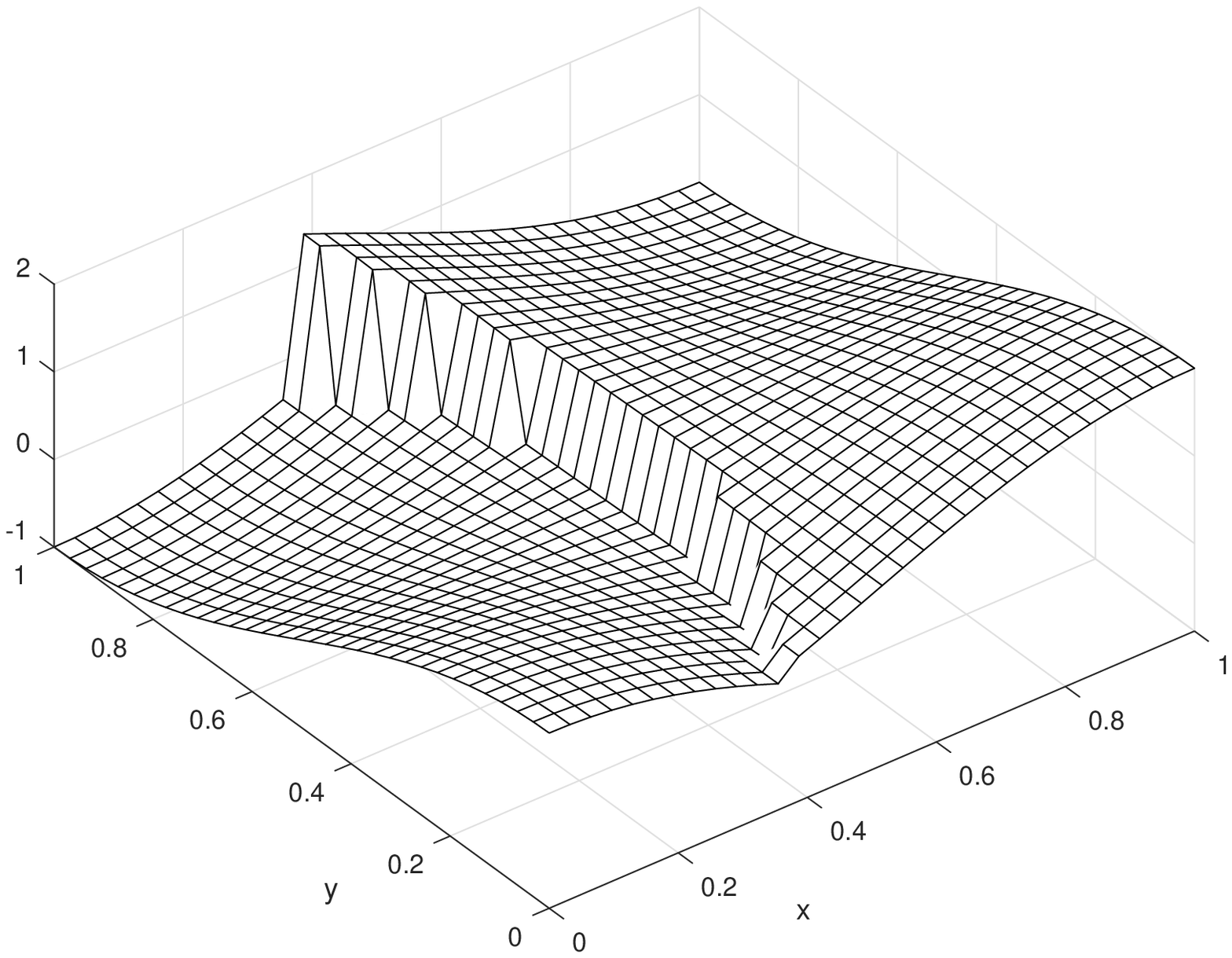,height=4cm}
\psfig{figure=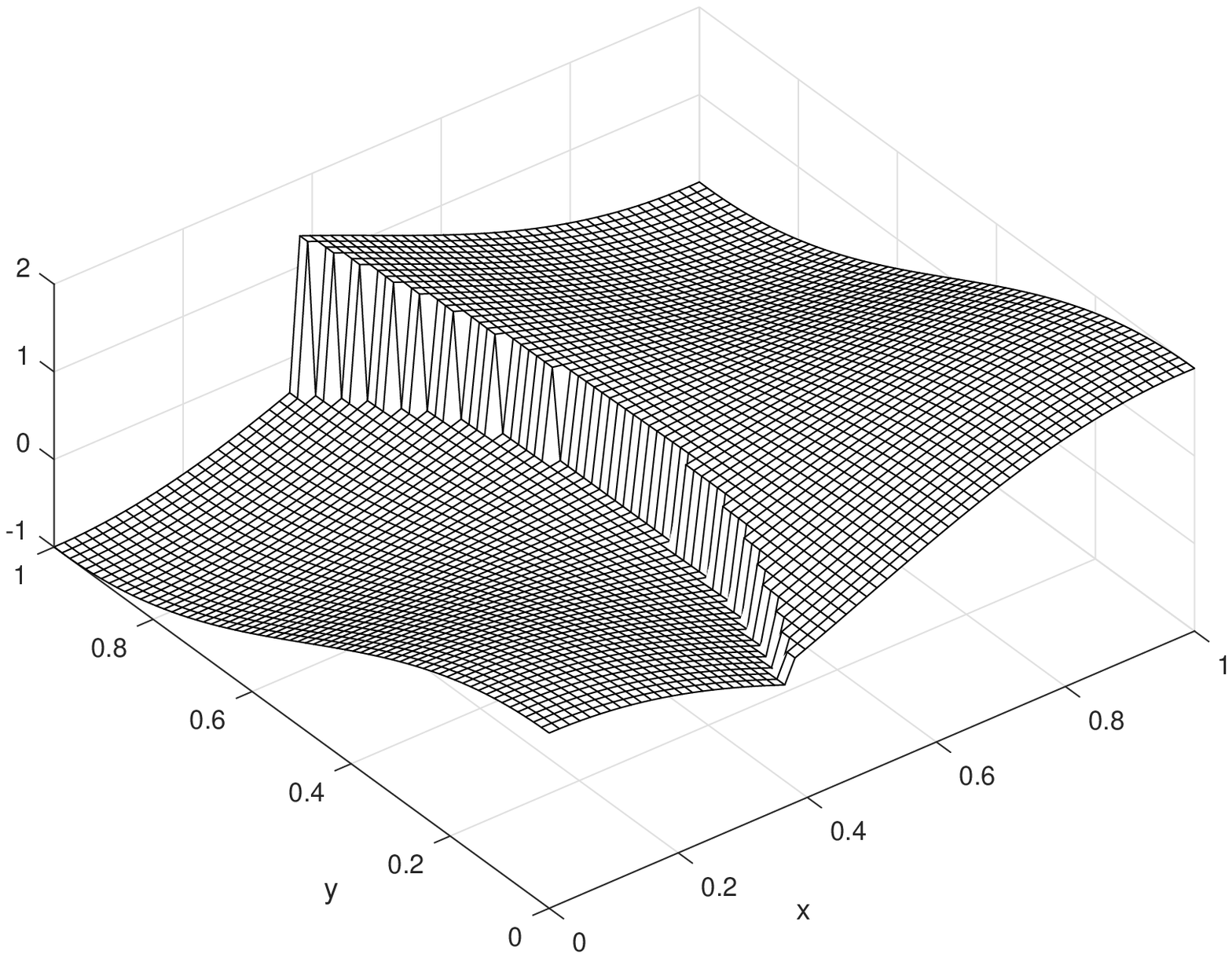,height=4cm}}
\centerline{\psfig{figure=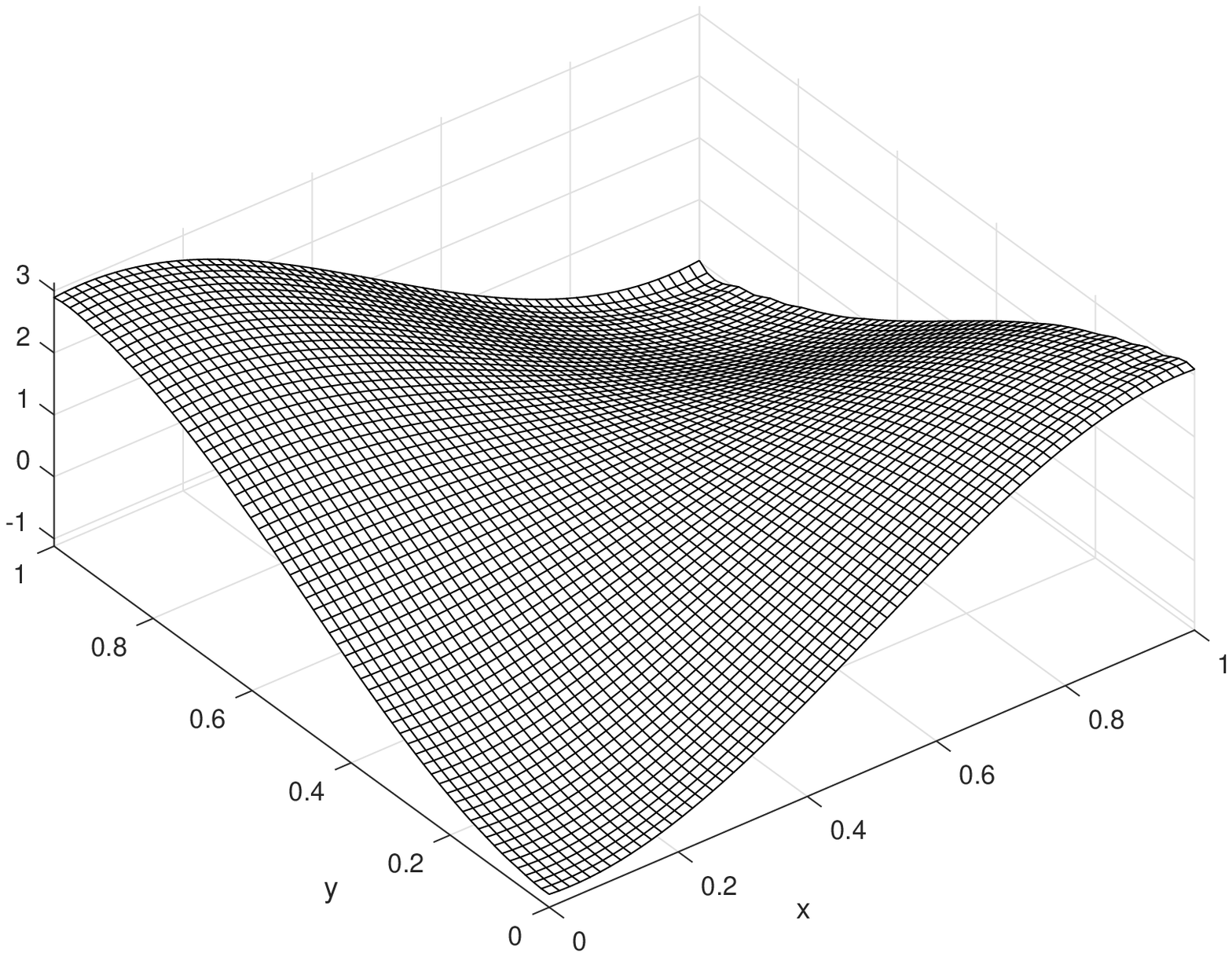,height=4cm}
\psfig{figure=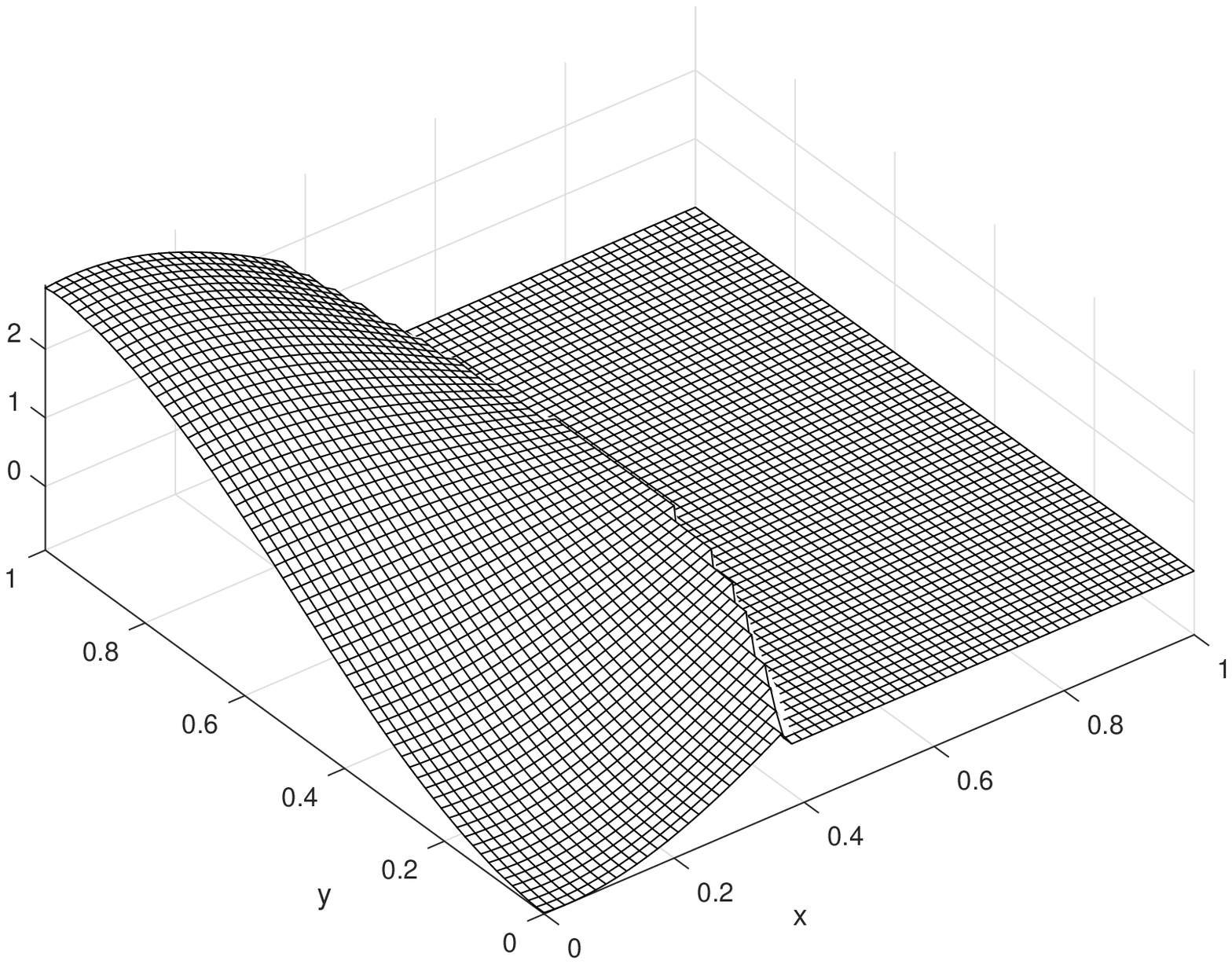,height=4cm}
}
\caption{Top to the left, plot of the function in (\ref{exp2D_point}). Top to the right, subdivided data. Bottom to the left subdivided correction term. Bottom to the right masked subdivided correction term.}\label{2D1}
\end{figure}

\section{Numerical results for the case of cell-averages data}\label{numexp_cell}

In this section we work with piecewise continuous functions supposing that the data is discretized by cell-averages, so that we are able to localize the position of the discontinuity up to the accuracy needed.

In all the experiments presented in this section we will use the following function discretized by cell-averages,
\begin{equation}\label{exp3}
f(x)=\left\{\begin{array}{ll}
10+\left(x-\frac{\pi}{6}\right)\left(x-\frac{\pi}{6}-10\right)+x^2+\sin(10x), & \textrm{if } x< \frac{\pi}{6},\\
x^2+\sin(10x),& \textrm{if } x\ge \frac{\pi}{6},
\end{array}\right.
\end{equation}
with $x\in[0, 1]$. It is easy to check that for the primitive $F$ of $f$, the jump in the function at $x=\frac{\pi}{6}$ is $[F(x=\pi/6)]=0$, the jump in the first derivative is $[F'(x=\pi/6)]=-[f(x=\pi/6)]=-10$ and in the second derivative is $[F''(x=\pi/6)]=[f'(x=\pi/6)]=10$.

Figure \ref{exp_cell} shows the limit function obtained by the linear algorithm (left), the quasi-linear algorithm (center) and the RC algorithm (right). In order to obtain these graphs we have started from 20 initial cell-averages of the function in (\ref{exp3}). When discretizing the function in (\ref{exp3}) through the cell-averages, we have represented the data at the lowest resolution $\bar{f}^{0}_j$ 
at the positions $x_{j-1}^{0}+\frac{h_{0}}{2}$. Figure \ref{exp_cell} shows that the linear algorithm produces oscillations close to the discontinuity. The RC and the quasi-linear algorithms do not produce oscillations and attain a very good approximation close to the discontinuities. Figure \ref{exp_cell_zoom} shows a zoom around the discon\-ti\-nuity. It is important to mention here that the point attained in the middle of the jump by the quasi-linear method and the RC algorithm is not due to di\-ffu\-sion introduced by the algorithms, it is due to the kind of discretization used. Mind that, if the function presents a discontinuity in the interval $({x_{j-1}^k}, x^k_{j})$, then the discretization in (\ref{discretization}) will always return a cell value at some point in the middle of the jump, that can be observed in the graphs of Figures \ref{exp_cell} and \ref{exp_cell_zoom}. This value simply corresponds to the mean of the function in the interval that contains the discontinuity, i.e. the cell-average value in (\ref{discretization}).

\begin{figure}[!ht]
\centerline{\psfig{figure=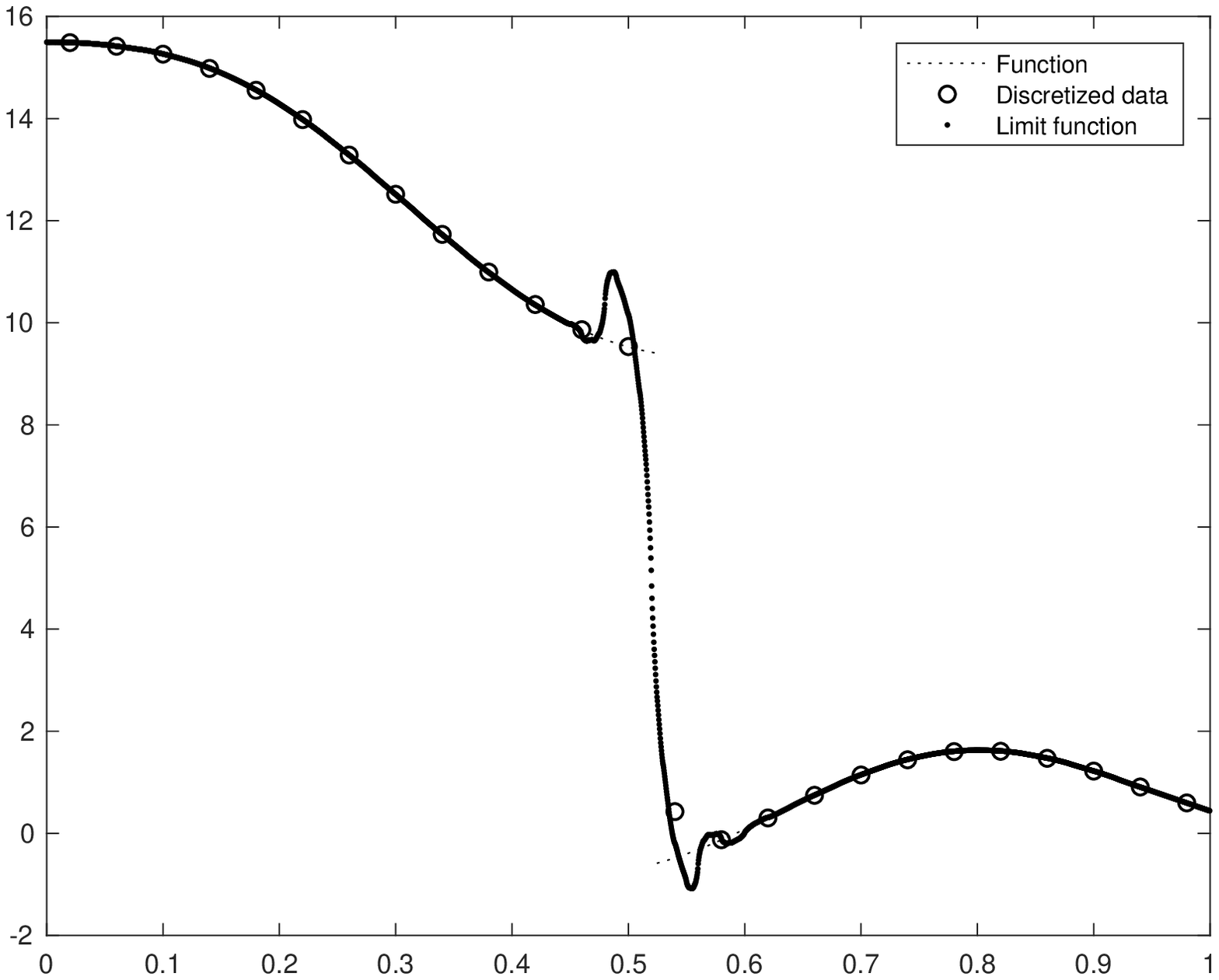,height=4cm}\\
\psfig{figure=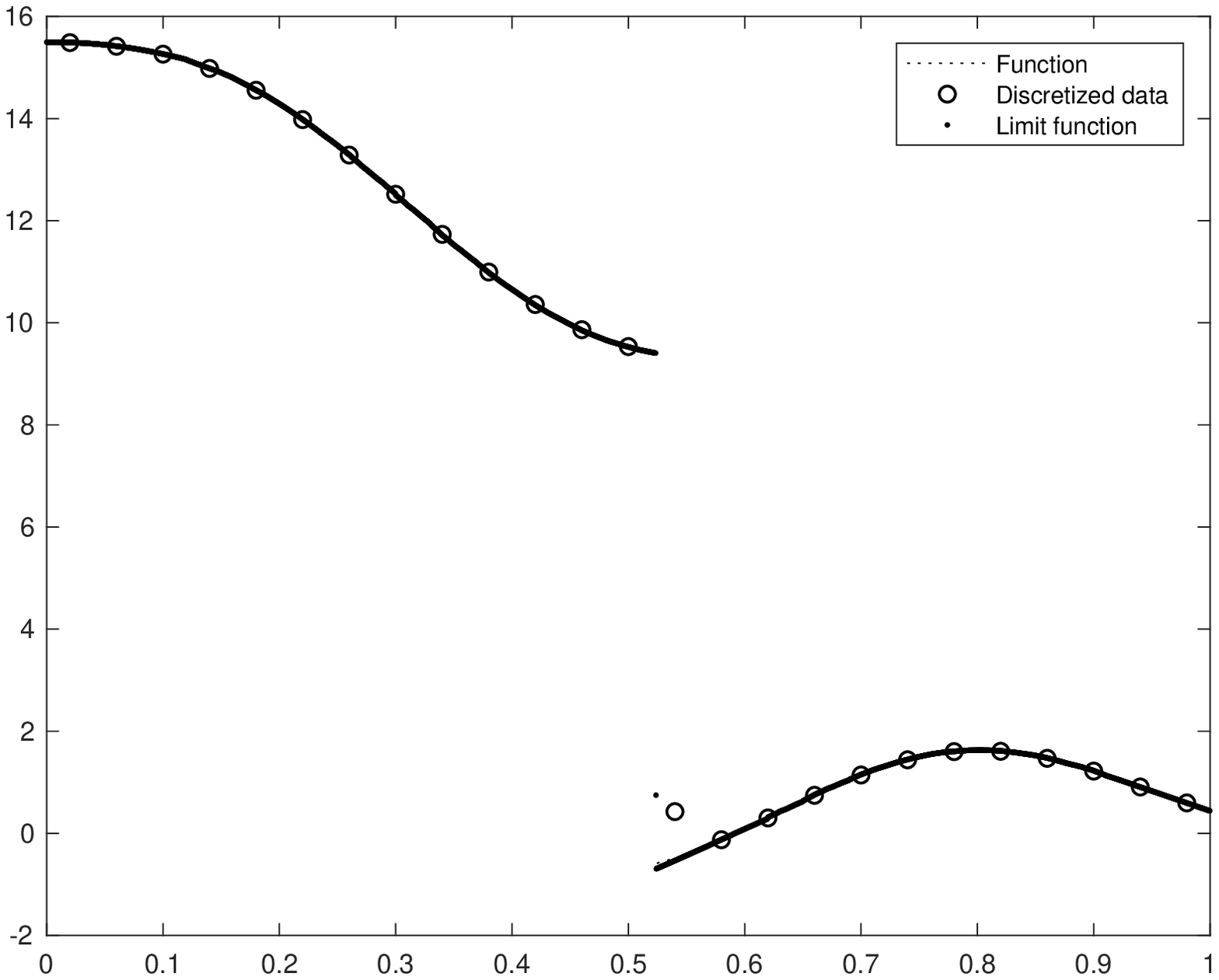,height=4cm}\\
\psfig{figure=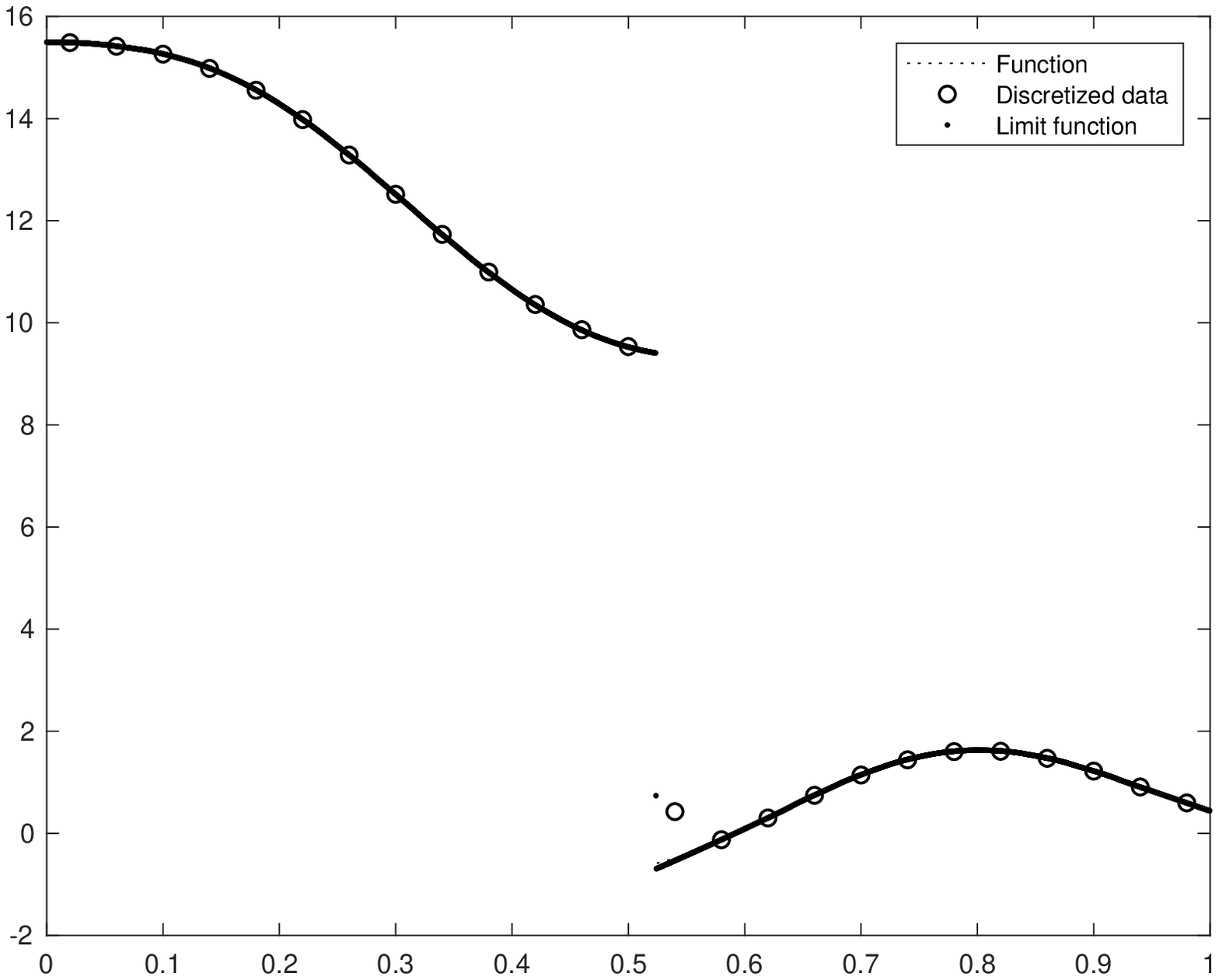,height=4cm}}
\caption{Limit function obtained by the linear algorithm (left), the quasi-linear algorithm (center) and the RC algorithm (right). In order to obtain these graphs, we have started from 20 initial cell-averages of the function in (\ref{exp1}).}\label{exp_cell}
\end{figure}

\begin{figure}[!ht]
\centerline{\psfig{figure=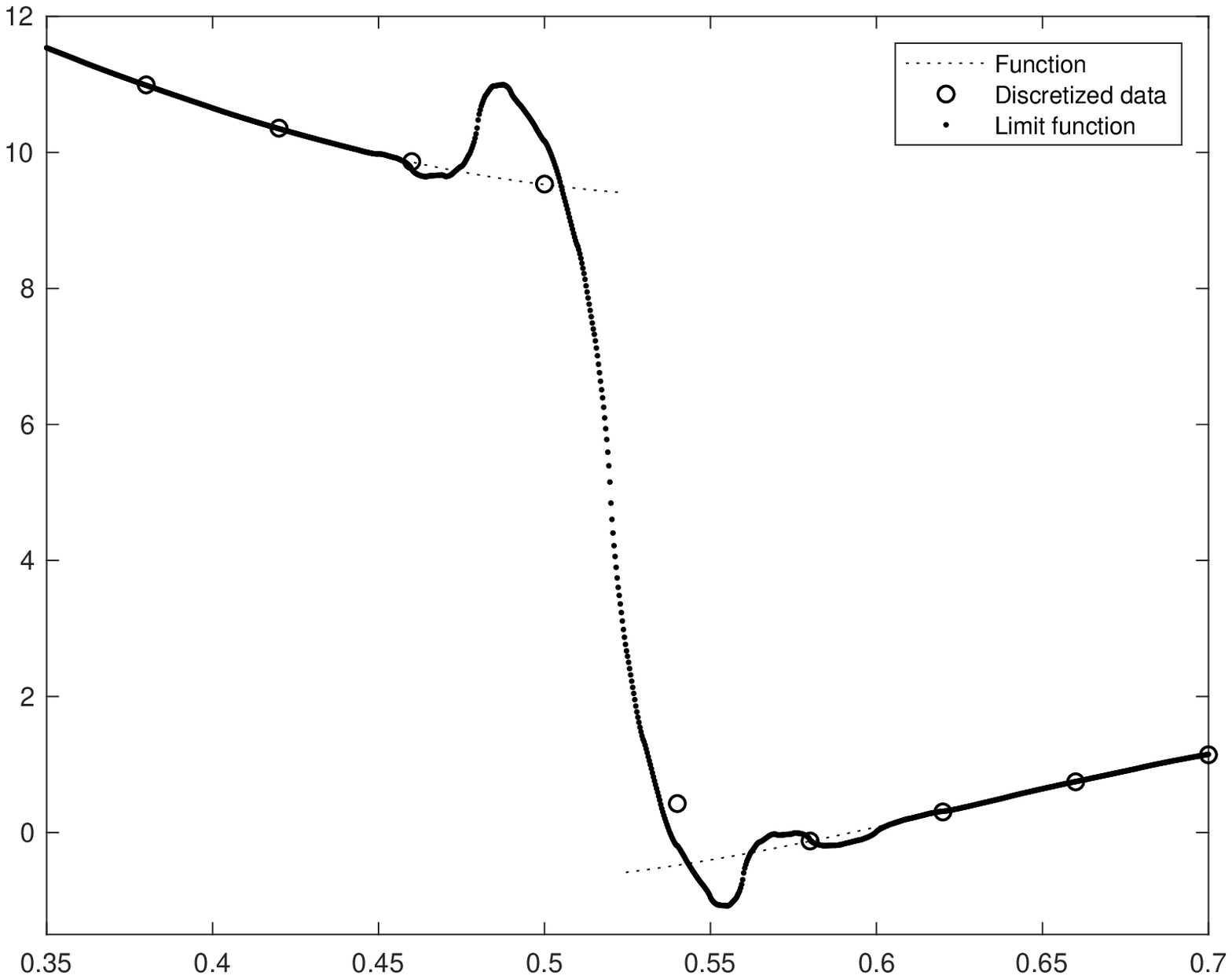,height=4cm}\\
\psfig{figure=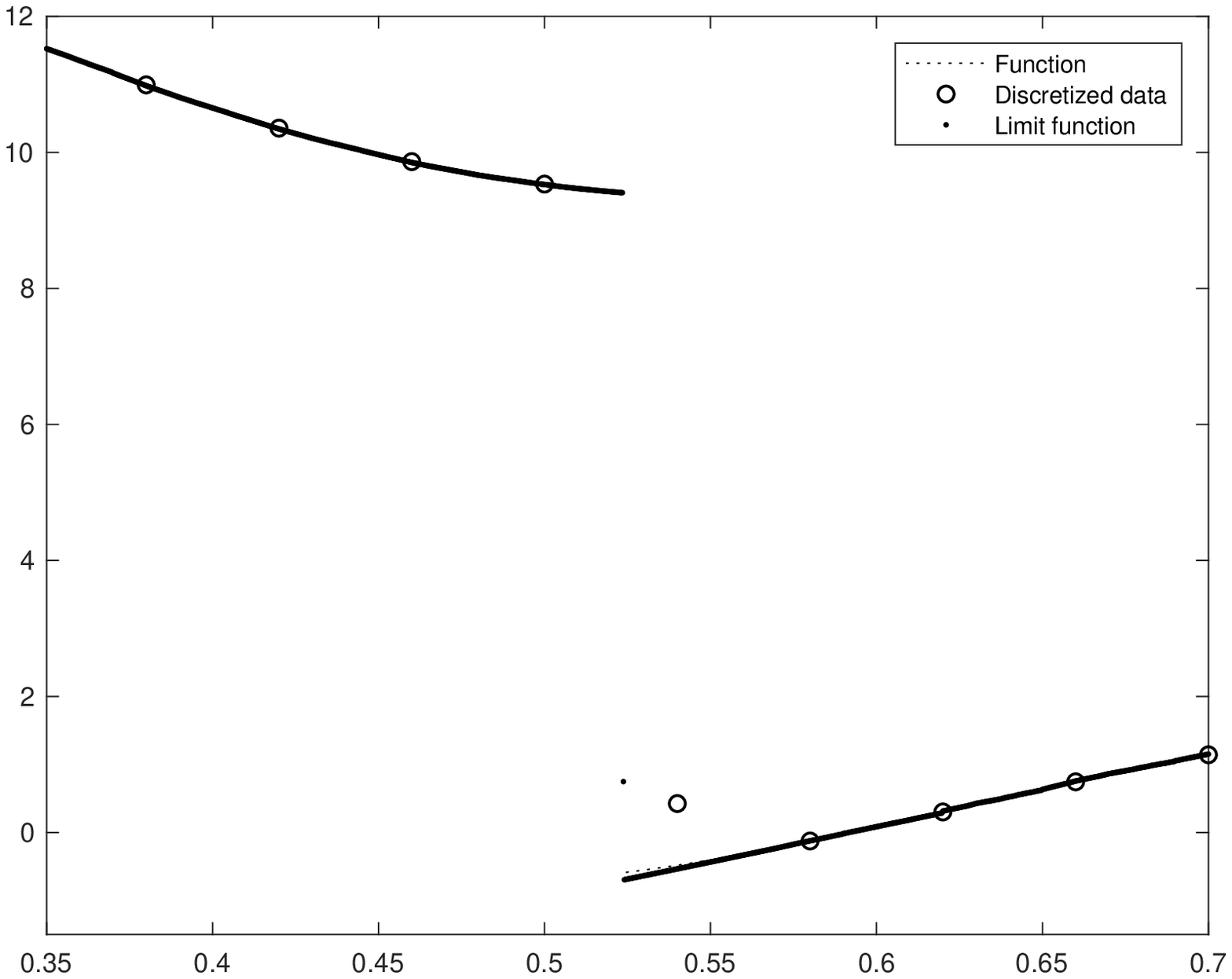,height=4cm}\\
\psfig{figure=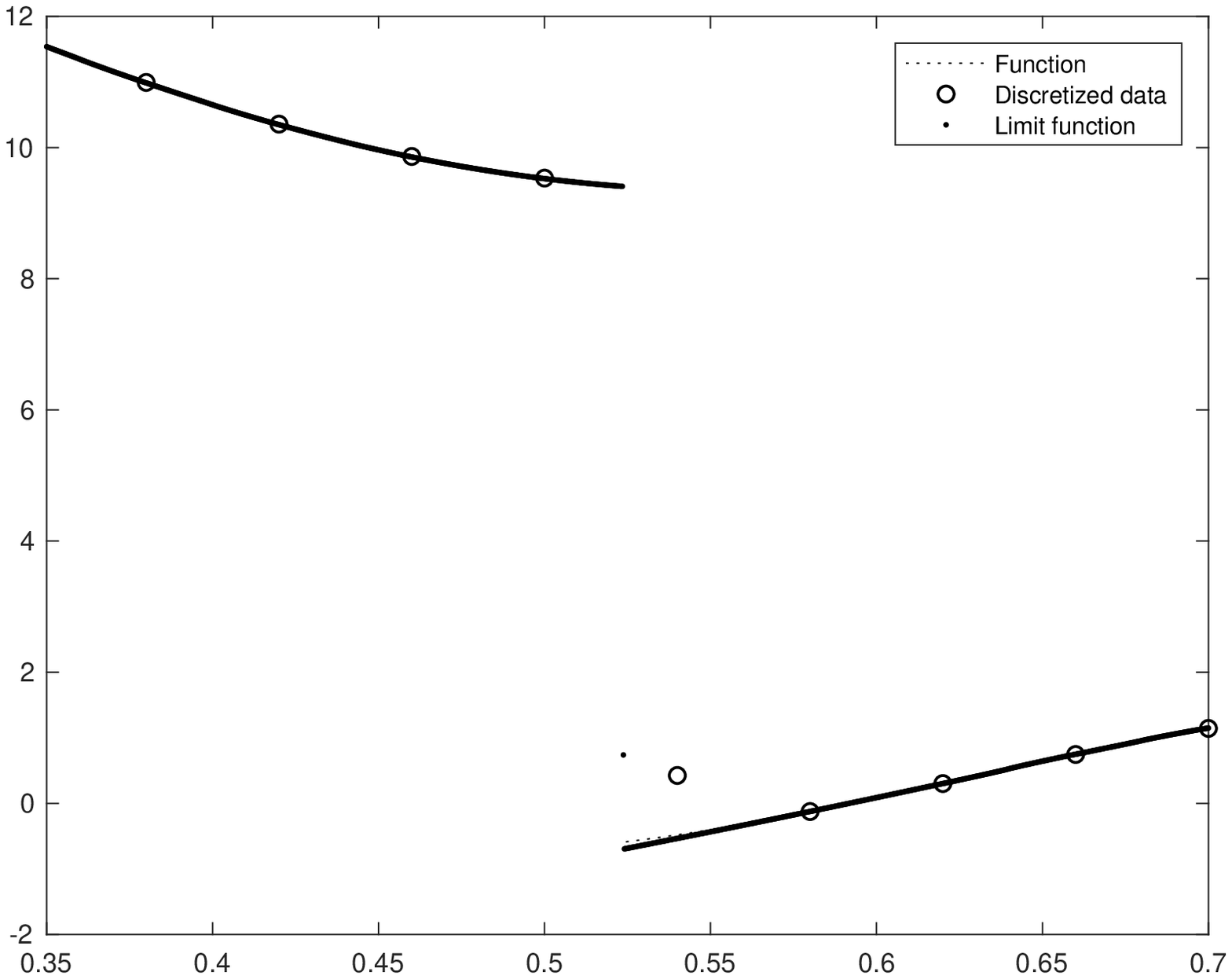,height=4cm}}
\caption{Zoom of the limit functions shown in Figure \ref{exp_cell} obtained by the linear algorithm (left), the quasi-linear algorithm (center) and the RC algorithm (right).}\label{exp_cell_zoom}
\end{figure}

\subsection{Numerical Regularity in the cell-average case}

In this subsection we analyze the numerical regularity attained by each of the algorithms for data discretized by cell-averages. In Table \ref{table4ch2} we present some numerical estimations of the regularity constant of the different algorithms analyzed. The data has been obtained from the discretization through cell-averages of the function in (\ref{exp3}). Following what was done in the point-values discretization, in order to obtain this table we start from $100$ initial data cells and we subdivide from $L=5$ to $L=10$ levels of subdivision in order to obtain an approximation of the limit function. We measure the numerical regularity for $x<\frac{\pi}{6}$, assuring that the discontinuity is not contained in the data. From this table we can see that
the numerical estimate of the re\-gularity for the RC algorithm is close to the one obtained by
the linear scheme. The quasi-linear scheme seems to be less regular.\\

\begin{table}[ht!]

\begin{center}
\resizebox{13cm}{!}{
\begin{tabular}{|c|c|c|c|c|c|c|c|c|}
\hline
&$L$ &  5 & 6&  7   & 8  & 9 & 10 \\
\hline
\multirow{ 3}{*}{$\beta_1$}
&Linear&  0.2240  &   0.8134  &    0.8348  &   0.8518  &   0.8656         &  {\bf  0.8771} \\\cline{2-8}
&Quasi-linear& -8.1221  &   0.5087 &   -0.7584  &   0.1382  &   0.3238      &   {\bf   -9.3917e-03}\\\cline{2-8}
&RC&   0.9981  &   0.9991  &   0.9995 &    0.9998    & 0.9999            &   {\bf    0.9999 }\\\cline{2-8}\hline
 \multirow{ 3}{*}{$\beta_2$}
&Linear& -0.7327  & 2.3471e-03 &   1.175e-03&  5.8784e-04 & 2.9401e-04    & {\bf  1.4704e-04} \\\cline{2-8}
&Quasi-linear& -12.2278  &  -0.4911  &    -1.7577   &  -0.8619  &  -0.67621          & {\bf  -1.0095}  \\\cline{2-8}
&RC&    0.2886 &    0.1270&    5.9099e-02  &  2.9719e-02  & 9.8887e-03      & {\bf   4.9282e-03}\\\cline{2-8}
\hline
\end{tabular}
}
\end{center}
\caption{ Numerical estimation of the limit functions regularity $C^{\beta_1-}$ and $C^{1+\beta_2-}$ for the different schemes presented and the function in (\ref{exp3}) discretized by cell-averages.}
 \label{table4ch2}
\end{table}

\subsection{Grid refinement analysis for the cell-averages sampling}
In this section we reproduce the grid refinement analysis that we performed for the point-values' case, using the infinity norm, but we also use the $L^1$ norm. 
For point-values data at the poins $\{x_j=jN_k^{-1}\}_{j=0, \cdots, N_k}$, we estimate $E^k_1$ using the cell-average data of the exact test function and its approximation on a mesh refined by a factor of $2^{-10}$.

In this case we have used the function presented in (\ref{exp3}) discretized by cell-averages. Table \ref{precision_cell} presents the errors and orders of approximation obtained by the three algorithms in the infinity norm. For our approach we show the error outside the interval $[x^*, s^*]$, in order to check the results esta\-blished in Theorem \ref{teocell}. For the linear and quasi-linear subdivision schemes we present the infinity norm in the whole domain. We can see how the linear algorithm losses the accuracy close to the discontinuity, while the quasi-linear algorithm and the RC approach keep high order of accuracy. In particular, the RC algorithm attains $O(h^3)$ accuracy, that is in accordance with the results of Theorem \ref{teocell}. We can also perform the same grid refinement analysis but using the $l^1$ norm instead. The results are presented in Table \ref{precision_cell_l1}. We can see in this table that we attain the order of accuracy expected. Mind that the accuracy has been reduced by one for all the algorithms as we are using a subdivision algorithm with a stencil of three cells. The RC approach attains the same order of accuracy as the quasi-linear algorithm.

\begin{table}[!ht]
\begin{center}
\resizebox{10cm}{!} {
\begin{tabular}{|c|c|c|c|c|c|c|c|c|c|c|c|c|c|c|}
\hline
&\multicolumn{2}{|c|}{RC algorithm}&\multicolumn{2}{|c|}{Linear}&\multicolumn{2}{|c|}{Quasi-linear}\\
\hline $N_k$ &$E^k_{\infty}$ & $order_k$&$E^k_{\infty}$ & $order_k$&$E^k_{\infty}$ & $order_k$
              \\
\hline  64& 1.2739e-02   &-& 5.1079 &0.3483 &6.0545e-01&2.3782    
            \\
\hline  128& 2.3556e-03&  2.4350 & 5.5371 & -0.1164& 2.9149e-01&  1.0546  
            \\
\hline  256&   5.9829e-04 &  1.9772   & 5.8782&  -0.0862&3.1212e-02&  3.2232    
            \\
\hline  5012 & 6.5693e-05 &3.1870  &  6.2890&  -0.0975& 3.2997e-03& 3.2417    
            \\
\hline  1024 & 7.3102e-06 &3.1678  &  6.5697&  {-0.0630} & 3.1356e-04&  {3.3955} 
            \\
\hline 2048& {\bf 7.8325e-07}&  {\bf3.3820}&  {\bf6.0839}&{\bf 0.1108}&{\bf 1.7553e-05}& {\bf 4.1589}
\\

\hline
\end{tabular}
}
\caption{Grid refinement analysis after ten levels of subdivision in the $l^{\infty}$ norm for the function in (\ref{exp3}) discretized by cell-averages and for the three subdivision schemes presented.}\label{precision_cell}
\end{center}
\end{table}

\begin{table}[!ht]
\begin{center}
\resizebox{10cm}{!} {
\begin{tabular}{|c|c|c|c|c|c|c|c|c|c|c|c|c|c|c|}
\hline
&\multicolumn{2}{|c|}{RC algorithm}&\multicolumn{2}{|c|}{Linear}&\multicolumn{2}{|c|}{Quasi-linear}\\
\hline $N_k$ &$E^k_1$ & $order_k$&$E^k_1$ & $order_k$&$E^k_1$ & $order_k$
              \\
\hline  64&    1.2052e-03 &-&1.2716e-01  &-&2.0014e-03  &-
            \\
\hline  128& 1.4370e-04  &3.0681 &8.0777e-02 &0.6546  &2.7861e-04  &2.8447
            \\
\hline  256&1.9401e-05 & 2.8889  & 2.2254e-02  &1.8599&3.8808e-05  & 2.8438
            \\
\hline  512& 2.0882e-06  &3.2158  &1.0952e-02  &  1.0229 &4.7043e-06 &  3.0443
            \\
\hline  1024 &2.4270e-07  &3.1050  &5.5043e-03  & 0.9925& 5.8186e-07  &3.0152
            \\
\hline  2048 &2.9298e-08  &{\bf3.0503} &3.1853e-03 &  {\bf0.7891} &7.2468e-08 &{\bf3.0053}
            \\

\hline
\end{tabular}
}
\caption{Grid refinement analysis in the $L^1$ norm for the function in (\ref{exp3}) discretized by cell-averages and for the three approximation schemes.}\label{precision_cell_l1}
\end{center}
\end{table}

\subsection{Approximation of bivariate cell-averages' data}\label{2D}
In this case we have applied the algorithms analyzed in previous sections to two-dimensional data using a tensor product approach, i.e. we directly process one-dimensional data by rows and then by columns.

In this section we will work with the bivariate function discretized by cell-averages,
\begin{equation}\label{exp2D}
f(x)=\left\{\begin{array}{ll}
\cos(\pi x)  \cos(\pi y), & \textrm{if } 0\le x<0.5, 0\le y<0.5,\\
-\cos(\pi x)  \cos(\pi y)+2,& \textrm{if } 0.5< x\le1, 0\le y<0.5,\\
-\cos(\pi x)  \cos(\pi y)+2,& \textrm{if } 0< x\le0.5, 0.5\le y<1,\\
-\cos(\pi x)  \cos(\pi y)+4,& \textrm{if } 0.5< x\le1, 0.5\le y<1.
\end{array}\right.
\end{equation}
Figure \ref{exp_2D} shows the result of one step of subdivision using the tensor product approach for the data presented in Figure \ref{exp_2D} top to the left. This means that we apply the one dimensional subdivision scheme by rows and then by columns. The result presented in Figure \ref{exp_2D} top to the right corresponds to the linear algorithm. We can observe that the effect of the discontinuity appears in the subdivided data in the form of diffusion and Gibbs effect. The result of the quasi-linear algorithm and the RC approach is presented in Figure \ref{exp_2D} bottom to the left and to the right respectively. We can see that both results are very similar. As shown in previous sections, the main difference is the regularity of the data close to the discontinuity, that is higher for the RC approach. In terms of accuracy, both algorithms perform similar.

\begin{figure}[!ht]
\centerline{\psfig{figure=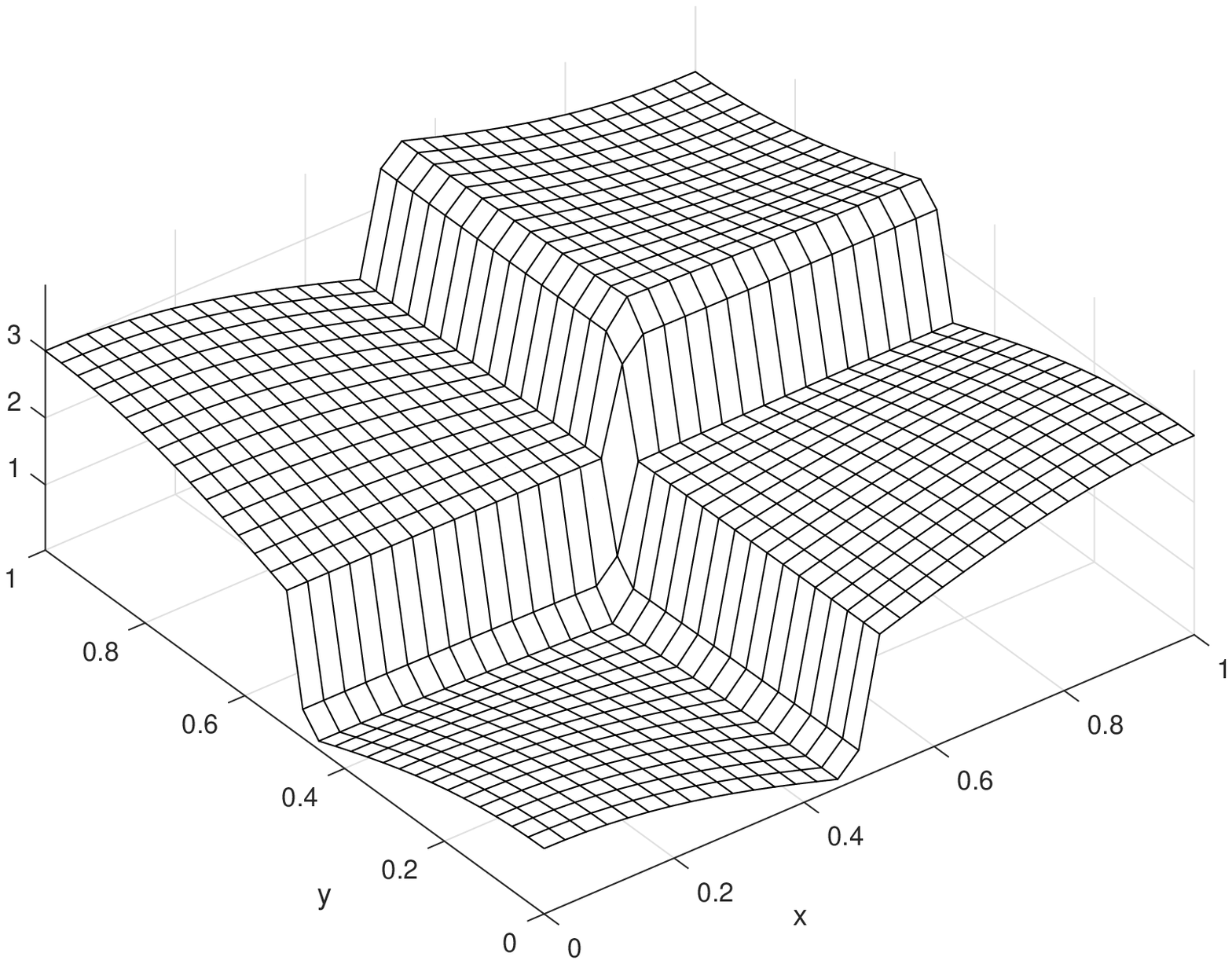,height=5cm}\\
\psfig{figure=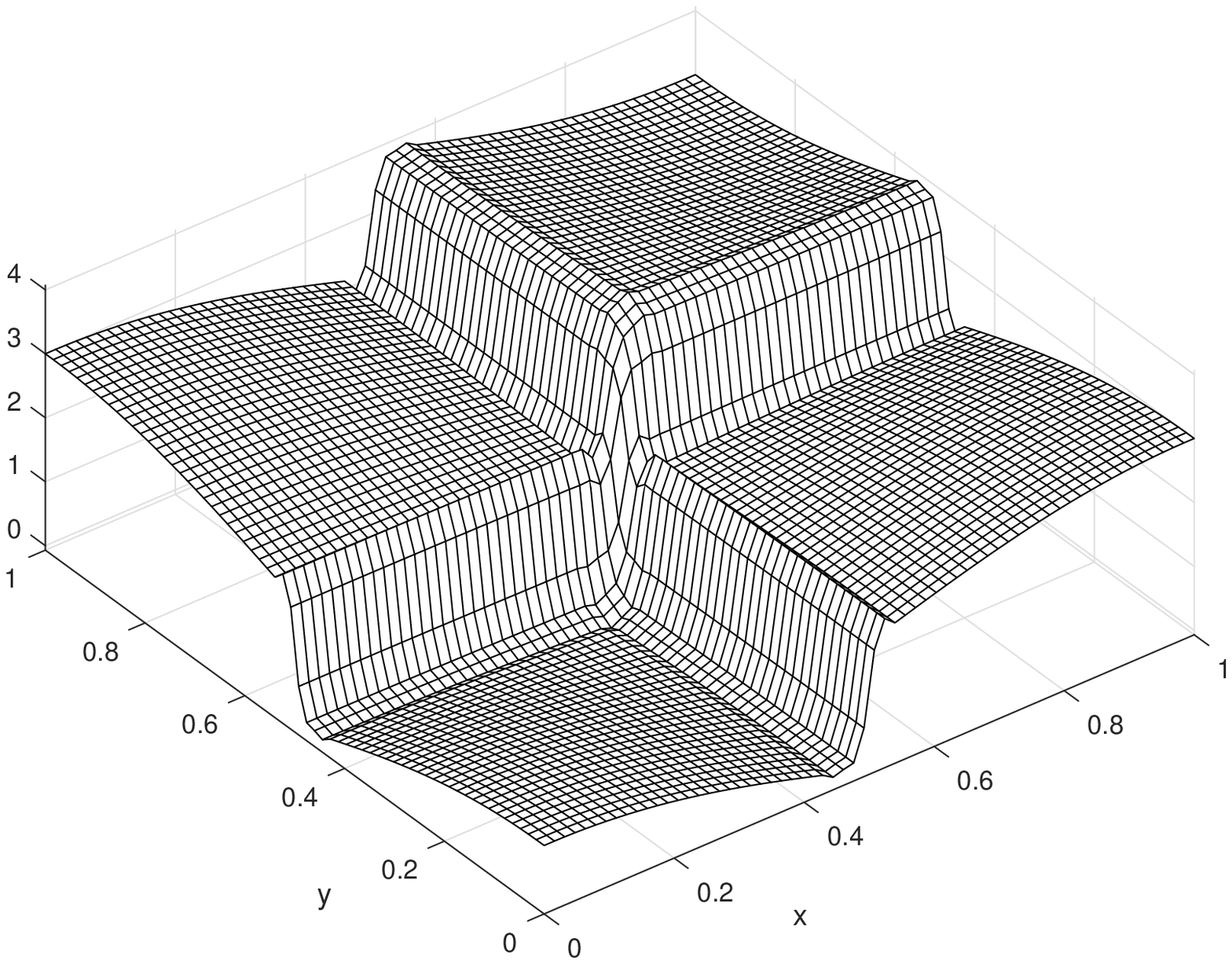,height=5cm}}
\centerline{\psfig{figure=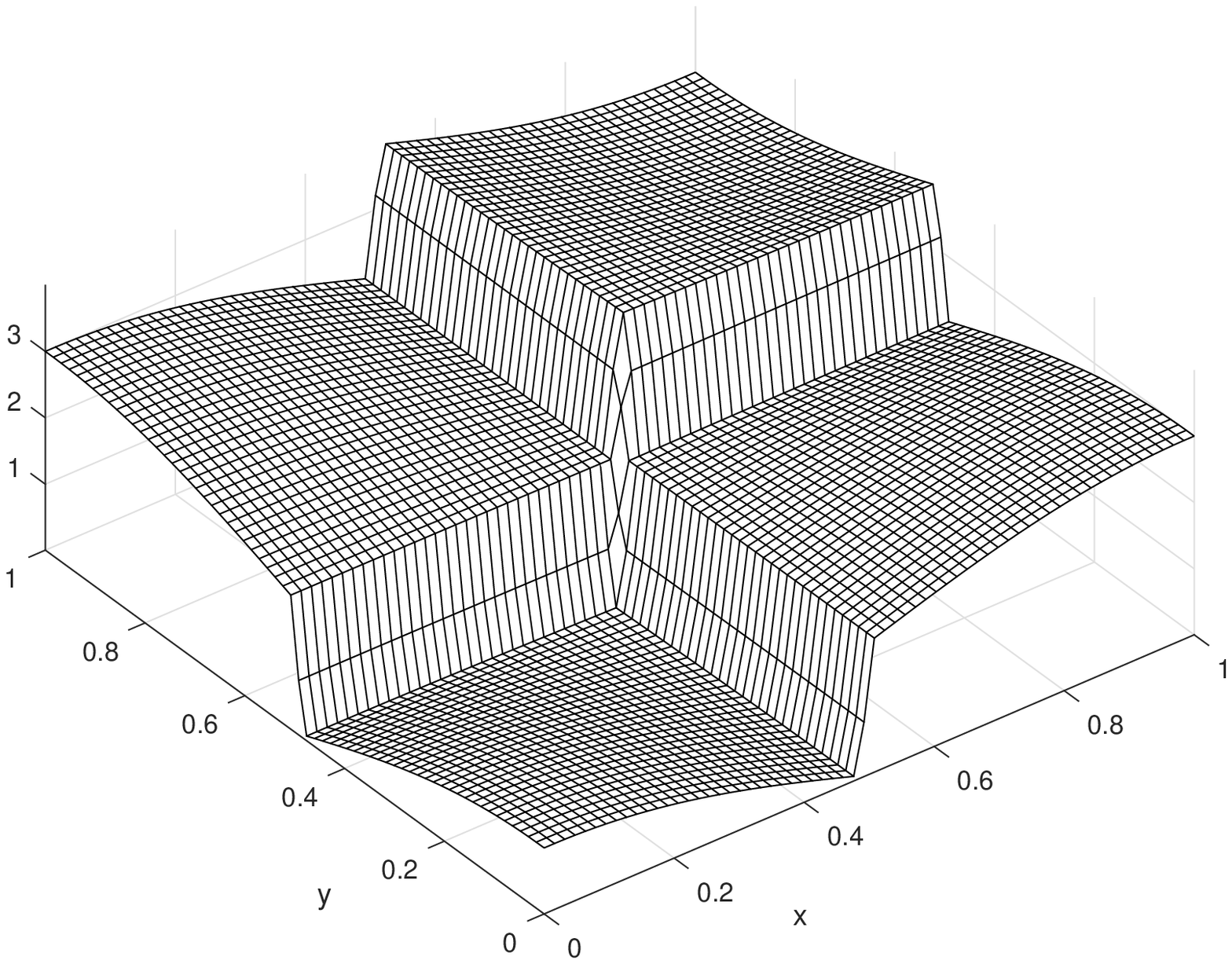,height=5cm}\\
\psfig{figure=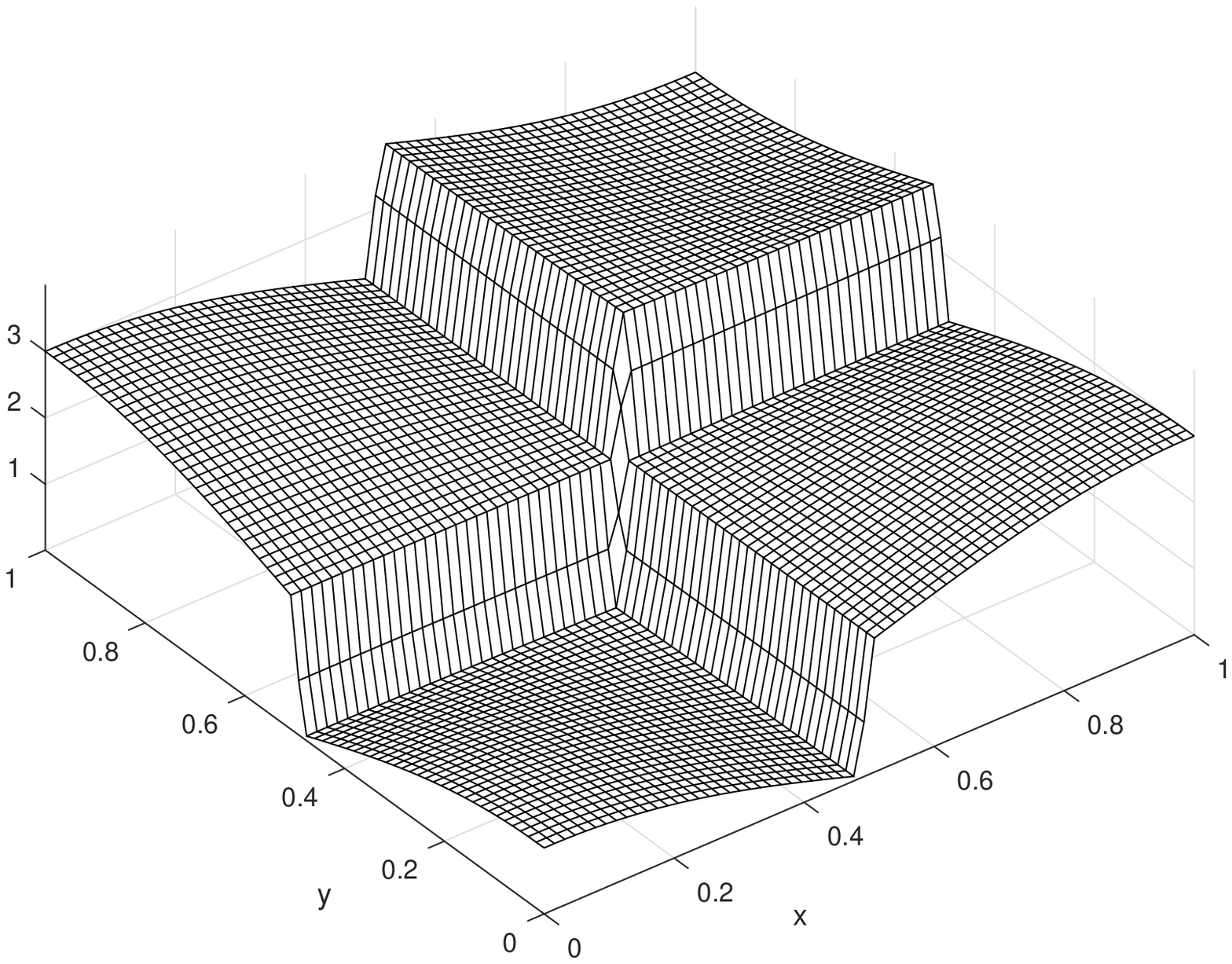,height=5cm}}
\caption{Top to the left, plot of the function in (\ref{exp2D_point}) discretized through cell-average values at a low resolution. Top to the right, subdivided data using the linear algorithm. Bottom to the left, subdivided data using the quasilinear algorithm. Bottom to the right, subdivided data using the RC algorithm.}\label{exp_2D}
\end{figure}

\section{Conclusions}\label{conclusion}

In this paper, we have introduced a regularization-correction approach to the problem of approximating piecewise smooth functions. In the first stage, the data is smoothed by subtracting an appropriate non-smooth data sequence. Then a uniform linear 4-point subdivision approximation operator is applied to the smoothed data. Finally, an approximation with the proper singularity structure is restored by correcting the smooth approximation with the non-smooth element used in the first stage. We deal with both cases of point-value data and cell-average data.
The resulting approximations for functions with discontinuities have the following five important properties,

1. Interpolation.

2. High precision.

3. High piecewise regularity.

4. No diffusion.

5. No oscillations.

We have used the 4-point Dubuc-Deslauriers interpolatory subdivision scheme \cite{DD} through which we obtain a $C^{2-}$ piecewise regular limit function that is capable of reproducing piecewise cubic polynomials.

The first advantage of our approach is that the resultant scheme presents the same regularity as the regularity of the linear subdivision scheme used in the second stage of the algorithm.
The accuracy of the regularization-correction approach is obtained from the accuracy reached in the location of the discontinuities and in the accuracy of the approximation of the jump in the function and its derivatives, which is done through Taylor's expansions. Thus, the corrected algorithm presents the same regularity, at each regula\-rity zone, as the linear subdivision scheme \cite{DL, DLG, DGL} plus the same accuracy as the quasi-linear scheme \cite{ACDD}. We present the results for the 4-point li\-near subdivision scheme, but the approach is indeed applicable to any other subdivision scheme aimed to deal with the approximation of functions with singularities. By construction, the corrected subdivision algorithm does not present the Gibbs phenomenon and does not introduce diffusion. As far as we know, this is the first time that an algorithm that owns all these properties at the same time appears in the literature. The numerical results confirm our theoretical analysis.

\end{document}